\documentclass[final]{siamltex}

\usepackage{amssymb,amsfonts,amsmath}

\usepackage[mathscr]{eucal}

\usepackage{stmaryrd}

\usepackage{amsbsy}

\usepackage{bm}

\usepackage{braket}

\newcommand{\Tra}{{\sf T}}

\newcommand{\Real}{\mathbb{R}}
\newcommand{\qtext}[1]{\quad\text{#1}\quad}

\newcommand{\V}[2][]{{\bm{#1 \mathbf{\MakeLowercase{#2}}}}} 
\newcommand{\M}[2][]{{\bm{#1 \mathbf{\MakeUppercase{#2}}}}} 
\newcommand{\T}[2][]{\boldsymbol{#1 \mathscr{\MakeUppercase{#2}}}} 
\newcommand{\TE}[3][]{#1 \MakeLowercase{#2}_{#3}}
\usepackage[usenames]{color}

\usepackage{stmaryrd} 

\usepackage{siunitx}
\usepackage{paralist}

\usepackage{subfig}

\usepackage{graphicx}
\usepackage{epstopdf}

\usepackage{placeins}

\usepackage{array}

\usepackage{algorithm}
\usepackage{algpseudocode}

\usepackage{environ}

\usepackage[colorlinks,urlcolor=blue,citecolor=blue,linkcolor=blue]{hyperref} 

\usepackage[notref,notcite]{showkeys}

\usepackage[noframe]{showframe}

\newcommand{\Rn}{\Real^n}
\newcommand{\Smn}[1][m,n]{\mathbb{S}^{[#1]}}
\newcommand{\TA}{\T{A}}
\newcommand{\TB}{\T{B}}
\newcommand{\Vy}{\V{y}}
\newcommand{\Vx}{\V{x}}

\newcommand{\SAxm}{\TA\Vx^m}
\newcommand{\VAxm}{\TA\Vx^{m-1}}
\newcommand{\MAxm}{\TA\Vx^{m-2}}
\newcommand{\SBxm}{\TB\Vx^m}
\newcommand{\VBxm}{\TB\Vx^{m-1}}
\newcommand{\MBxm}{\TB\Vx^{m-2}}

\newcommand{\VExm}{\T{E}\Vx^{m-1}}
\newcommand{\MExm}{\T{E}\Vx^{m-2}}
\newcommand{\Mxxt}{\V{x}\V{x}^{\Tra}}
\newcommand{\Msop}[2]{#1 \circledcirc #2}

\newcommand{\Snxm}[1][m]{\|\Vx\|^{#1}}
\newcommand{\Sfx}[1][(\Vx)]{f#1}
\newcommand{\Sfhatx}[1][(\Vx)]{\hat f#1}
\newcommand{\Vgx}[1][(\Vx)]{\V{g}#1}
\newcommand{\Vghatx}[1][(\Vx)]{\V[\hat]{g}#1}
\newcommand{\MHx}[1][(\Vx)]{\M{H}#1}
\newcommand{\MHhatx}[1][(\Vx)]{\M[\hat]{H}#1}

\newcommand{\ParensX}{(\Vx)}
\newcommand{\FuncA}{f_1\ParensX}
\newcommand{\FuncB}{f_2\ParensX}
\newcommand{\FuncC}{f_3\ParensX}
\newcommand{\GradA}{\V{g}_1\ParensX}
\newcommand{\GradB}{\V{g}_2\ParensX}
\newcommand{\GradC}{\V{g}_3\ParensX}
\newcommand{\HessA}{\M{H}_1\ParensX}
\newcommand{\HessB}{\M{H}_2\ParensX}
\newcommand{\HessC}{\M{H}_3\ParensX}

\newtheorem{Example}[theorem]{Example}
\newenvironment{example}{\begin{Example}\rm}{~~$\square$\end{Example}}
\newenvironment{exampleplusname}[1]{\begin{Example}[#1]\rm}{~~$\square$\end{Example}}

\newcolumntype{R}{>{$}r<{$}}  %
\newcolumntype{V}[1]{>{[\;}*{#1}{R@{\;\;}}R<{\;]}} %
\newcolumntype{E}[1]{*{#1}{R@{,\;}}R} %

\usepackage{calc}
\newcounter{dimnmo}
\newcounter{dimn}
\NewEnviron{ResultsSubTable}[2]{  
  \setcounter{dimn}{#2}
  \setcounter{dimnmo}{\value{dimn}-1}
  \vspace{-2mm}
  \subfloat[#1]{
    \begin{tabular}{|c|R|V{\thedimnmo}|r|
        r@{\,}>{\text{(}}S<{\text{)}}|S@{\,$\pm$\,}S|r@{\,$\pm$\,}l|}
      \hline
      occ. & 
      \multicolumn{1}{c|}{$\lambda$} &     
      \multicolumn{\thedimn}{c|}{$\Vx$} & 
      its. &  
      \multicolumn{2}{c|}{viol.\@ (max.)} & 
      \multicolumn{2}{c|}{error} &
      \multicolumn{2}{c|}{time (sec.)}
      \\ \hline
      \BODY
    \end{tabular}}}
{}

\NewEnviron{HanResultsSubTable}[2]{  
  \setcounter{dimn}{#2}
  \setcounter{dimnmo}{\value{dimn}-1}
  \vspace{-2mm}
  \subfloat[#1]{
    \begin{tabular}{|c|R|V{\thedimnmo}|r|
        S@{\,$\pm$\,}S|r@{\,$\pm$\,}l|}
      \hline
      occ. & 
      \multicolumn{1}{c|}{$\lambda$} &     
      \multicolumn{\thedimn}{c|}{$\Vx$} & 
      fevals &  
      \multicolumn{2}{c|}{error} &
      \multicolumn{2}{c|}{time (sec.)}
      \\ \hline
      \BODY
    \end{tabular}}}
{}

\newcommand{\EpHeader}[2]{%
  \multicolumn{1}{|c|}{$\lambda$} &
  \multicolumn{#1}{c|}{$\Vx^{\Tra}$} &
  \multicolumn{#2}{c|}{$\M{C}(\lambda,\Vx)$ evals.} &
  \multicolumn{1}{c|}{type}
}  

\newcommand{\Sec}[1]{\hyperref[sec:#1]{\S\ref*{sec:#1}}} %
\newcommand{\App}[1]{\hyperref[sec:#1]{Appendix~\ref*{sec:#1}}} %
\newcommand{\Eqn}[1]{\hyperref[eq:#1]{(\ref*{eq:#1})}} %
\newcommand{\Part}[1]{\hyperref[#1]{(\ref*{#1})}} %
\newcommand{\Fig}[1]{\hyperref[fig:#1]{Figure~\ref*{fig:#1}}} %
\newcommand{\Tab}[1]{\hyperref[tab:#1]{Table~\ref*{tab:#1}}} %
\newcommand{\Thm}[1]{\hyperref[thm:#1]{Theorem~\ref*{thm:#1}}} %
\newcommand{\Lem}[1]{\hyperref[lem:#1]{Lemma~\ref*{lem:#1}}} %
\newcommand{\Prop}[1]{\hyperref[prop:#1]{Property~\ref*{prop:#1}}} %
\newcommand{\Cor}[1]{\hyperref[cor:#1]{Corollary~\ref*{cor:#1}}} %
\newcommand{\Def}[1]{\hyperref[def:#1]{Definition~\ref*{def:#1}}} %
\newcommand{\Alg}[1]{\hyperref[alg:#1]{Algorithm~\ref*{alg:#1}}} %
\newcommand{\Ex}[1]{\hyperref[ex:#1]{Example~\ref*{ex:#1}}} %

\begin{document}
 
\title{An Adaptive Shifted Power Method for Computing Generalized Tensor Eigenpairs%
  \thanks{This work was funded by the applied mathematics program at
    the U.S.  Department of Energy and by an Excellence Award from the
    Laboratory Directed Research \& Development (LDRD) program at
    Sandia National Laboratories. Sandia National Laboratories is a
    multiprogram laboratory operated by Sandia Corporation, a wholly
    owned subsidiary of Lockheed Martin Corporation, for the United
    States Department of Energy's National Nuclear Security
    Administration under contract DE-AC04-94AL85000.}}

\author{Tamara G. Kolda\footnotemark[2] \and Jackson R. Mayo\footnotemark[2]} 
\maketitle

\renewcommand{\thefootnote}{\fnsymbol{footnote}}
\footnotetext[2]{Sandia National Laboratories, Livermore, CA. Email: \{tgkolda,jmayo\}@sandia.gov.}
\renewcommand{\thefootnote}{\arabic{footnote}}

\begin{abstract}
  Several tensor eigenpair definitions have been put forth in the past
  decade, but these can all be unified under generalized tensor
  eigenpair framework, introduced by Chang, Pearson, and Zhang (2009).
  Given $m$th-order, $n$-dimensional real-valued symmetric tensors
  $\TA$ and $\TB$, the goal is to find $\lambda \in \Real$ and $\Vx
  \in \Rn, \Vx \neq 0$ such that $\VAxm = \lambda \VBxm$. Different
  choices for $\TB$ yield different versions of the tensor eigenvalue
  problem. We present our generalized eigenproblem adaptive power
  (GEAP) method for solving the problem, which is an extension
  of the shifted symmetric higher-order power method (SS-HOPM) for
  finding Z-eigenpairs. A major drawback of SS-HOPM was that its
  performance depended in choosing an appropriate shift, but our GEAP
  method also includes an adaptive method for choosing the shift
  automatically.
\end{abstract}
 

\pagestyle{myheadings}
\thispagestyle{plain}
\markboth{\sc T.~G.~Kolda and J.~R.~Mayo}%
{\sc Adaptive Shifted Power Method for Computing Generalized Tensor Eigenpairs}

\section{Introduction} 
\label{sec:introduction} 
Suppose $\TA$ is a real-valued, $m$th-order, $n$-dimensional tensors and
$\Vx$ is a real-valued $n$-vector. We let $\VAxm$ denote the
$n$-vector defined by
\begin{displaymath}
  (\VAxm)_{i_1} = 
  \sum_{i_2 = 1}^n \cdots \sum_{i_m=1}^n a_{i_1 \dots i_m} x_{i_2} \cdots x_{i_m}
  \qtext{for} i_1 = 1,\dots,n.
\end{displaymath}
We let $\SAxm$ denote the scalar defined by $\SAxm = \Vx^\Tra
(\VAxm)$. We say the tensor $\TA$ is \emph{symmetric} if its entries
are invariant under permutation. We say the tensor $\TA$ is
\emph{positive definite} if $\SAxm > 0$ for all $\Vx \neq 0$. 

The notion of generalized eigenpairs has been defined for tensors by
Chang, Pearson, and Zhang \cite{ChPeZh09} as follows.  Let $\TA$ and
$\TB$ be real-valued, $m$th-order, $n$-dimensional symmetric tensors.
Assume further that $m$ is even and $\TB$ is positive definite.  We
say $(\lambda,\Vx) \in \Real \times \{ \Rn \setminus \set{\V{0}}\}$ is
a \emph{generalized eigenpair} (also known as a
\emph{$\TB$-eigenpair}) if
\begin{equation}
  \label{eq:gtep}
  \VAxm = \lambda \VBxm.
\end{equation}
Taking the dot product with $\Vx$, it is clear that any solution
satisfies
\begin{equation}
  \label{eq:lambda}
  \lambda = \frac{\SAxm}{\SBxm}.
\end{equation}

The advantage of the generalized eigenpair framework is that it nicely
encapsulates multiple definitions of tensor eigenvalues, as follows.
\begin{asparaitem}
  
  \item A \emph{Z-eigenpair} \cite{Qi05,Li05} is defined as a pair 
  $(\lambda,\Vx) \in \Real \times \Rn$  
  such that
  \begin{equation}
  \label{eq:ztep}
  \VAxm = \lambda \Vx \qtext{and} \Snxm[] = 1.
  \end{equation}
  This is equivalent to a generalized tensor eigenpair with $\TB = \T{E}$, 
  the identity tensor such that $\VExm   = \Snxm[m-2] \Vx$   
  for all $\Vx \in \Rn$ \cite{ChPeZh09}. 
  Note that, unlike ordinary tensor Z-eigenpairs, generalized tensor eigenpairs   allow arbitrary rescaling of the eigenvector $\Vx$ with no effect on the   eigenvalue $\lambda$. In this way, the generalized tensor eigenvalue problem   preserves the homogeneity of the corresponding matrix eigenproblem.
  
  \item
  An \emph{H-eigenpair} is defined as a pair $(\lambda,\Vx) \in \Real \times \{ \Rn \setminus \set{\V{0}}\}$ such that
  \begin{equation}
  \label{eq:htep}
  \VAxm = \lambda \Vx^{[m-1]}.
  \end{equation}
  Here $\Vx^{[m-1]}$ denotes elementwise power, i.e., $(\Vx^{[m-1]})_i
  \equiv \Vx_i^{m-1}$, for $i=1,\dots,n$.
  This is equivalent to a generalized tensor eigenpair with $b_{i_1i_2\dots i_m} = \delta_{i_1i_2\dots i_m}$ \cite{ChPeZh09}.
  
  \item Let $\M{D}$ be a symmetric $n \times n$ matrix and assume $m=4$.
  We say $(\lambda, \Vx)$ is a \emph{D-eigenpair} \cite{QiWaWu08} if
  \begin{displaymath}
  \VAxm = \lambda \M{D}\Vx \qtext{and} \Vx^{\Tra}\M{D}\Vx = 1.
  \end{displaymath}
  This is equivalent to a $\TB$-eigenpair where $\TB$ is the symmetrized 
  tensor outer product of $\M{D}$ with itself \cite{ChPeZh09}.
\end{asparaitem}

In this paper, we describe a method for computing generalized eigenpairs. Our
method is a generalization of the shifted symmetric higher-order power method
(SS-HOPM) that we previously introduced for computing Z-eigenvalues
\cite{KoMa11}. In addition to
generalizing the method, we have also 
significantly improved it by adding an adaptive method for
choosing the shift. To derive the method, we reformulate the generalized eigenproblem, 
\Eqn{gtep}, as a  nonlinear program
such that any generalized eigenpair is equivalent to a KKT point 
in \Sec{formulation}. We develop 
an adaptive, monotonically convergent, shifted power method for solving the 
optimization problem in \Sec{derivation}. We call our method the 
Generalized Eigenproblem Adaptive Power (GEAP) method.
In \Sec{experiments}, 
we show that the GEAP method is much faster than the SS-HOPM method 
for finding Z-eigenpairs due to its adaptive shift selection.
Furthermore, the GEAP method is shown to find other types of generalized 
eigenpairs, by illustrating it on examples from related literature 
as well as a randomly generated example. 
This is the only known method for finding generalized eigenpairs 
besides direct numerical solution; we survey related work in \Sec{related}.

\section{Notation and preliminaries}
\label{sec:notation}  

A symmetric tensor has entries that are invariant under any permutation of its
indices. More formally, a real-valued, $m$th-order, $n$-dimensional tensor $\TA$
is \emph{symmetric} if
\begin{displaymath}
  \TE{a}{i_{p(1)} \cdots i_{p(m)}} = \TE{a}{i_1 \cdots i_m}
  \qtext{for all} i_1, \dots, i_m \in \{1,\dots,n\}
  \qtext{and} p \in \Pi_{m},
\end{displaymath}
where $\Pi_m$ denotes the space of all $m$-permutations. We let $\Smn$ denote
the space of all symmetric, real-valued, $m$th-order, $n$-dimensional tensors.

Let $\TA \in \Smn$, then we can define the following tensor-vector products.
\begin{align}
  \label{eq:Axm}
  \SAxm &= \sum_{i_1=1}^n \cdots \sum_{i_m=1}^n 
  a_{i_1 i_2 \dots i_m} x_{i_1} \cdots x_{i_m}, \\
  \label{eq:Axm1}
  (\VAxm)_{i_1} &= \sum_{i_2=1}^n \cdots \sum_{i_m=1}^n 
  a_{i_1 i_2 \dots i_m} x_{i_2} \cdots x_{i_m} 
  & \text{for all } & i_1 = 1,\dots,n,\\
  \label{eq:Amx2}
  (\MAxm)_{i_1i_2} &= \sum_{i_3=1}^n \cdots \sum_{i_m=1}^n 
  a_{i_1 i_2 \dots i_m} x_{i_3} \cdots x_{i_m} 
  & \text{for all } & i_1,i_2 = 1,\dots,n. 
\end{align}
Observe that the derivatives of the tensor-vector product w.r.t.\@ $\Vx$ are given by
\begin{align*}
  \nabla (\SAxm) &= m \VAxm, &
  \nabla^2 (\SAxm) &= m(m-1) \MAxm. 
\end{align*}
We say a tensor $\TA \in \Smn$ is \emph{positive definite} if
\begin{displaymath}
\SAxm > 0 \qtext{for all} \Vx \in \Rn, \Vx \neq 0.
\end{displaymath}
We let $\Smn_+$ denote the space of positive definite tensors in
$\Smn$.

We use the symbol $\circledcirc$ to mean symmetrized outer product, i.e., 
\begin{displaymath}
\Msop{\V{a}}{\V{b}} = \V{a}\V{b}^{\Tra} + \V{b}\V{a}^{\Tra}.
\end{displaymath}

\section{Problem reformulation}
\label{sec:formulation}

Let $\Sigma$ denote the unit sphere, i.e., 
\begin{displaymath}
  \Sigma = \set{\Vx \in \Rn | \Snxm = 1 }.
\end{displaymath}
Let $\TA \in \Smn$ and $\TB \in \Smn_+$. Then we may define the
nonlinear program
\begin{equation}
\label{eq:opt}
\max \Sfx = \frac{\SAxm}{\SBxm} \Snxm
\qtext{subject to} \Vx \in \Sigma.
\end{equation}
The constraint makes the term $\Snxm$ in $\Sfx$ superfluous;
nevertheless, we retain this form since choosing $\TB = \T{E}$ yields
$f(x) = \SAxm$, as in \cite{KoMa11}.

The details of computing the derivatives are provided in
\App{derivatives}. Here we simply state the results as a theorem.
\begin{theorem}
Let $\Sfx$ be as defined in \Eqn{opt}. 
For $\Vx \in \Sigma$, the gradient is
\begin{equation}\label{eq:grad}
  \Vgx \equiv \nabla \Sfx 
  = \frac{m}{\SBxm} 
  \Biggl[
  \bigl( \SAxm \bigr) \; \Vx 
  \;+\; \VAxm
  \;-\; \biggl( \frac{\SAxm}{\SBxm} \biggr) \; \VBxm
  \Biggr].
\end{equation}
For $\Vx \in \Sigma$, the Hessian is
\begin{multline} \label{eq:hess}
  \MHx \equiv \nabla^2 \Sfx = 
  \frac{m^2 \SAxm}{(\SBxm)^3} \bigl( \Msop{\VBxm}{\VBxm} \bigr) \\
  + \frac{m}{\SBxm} 
  \biggl[
  (m-1) \MAxm 
  + \SAxm \bigl( \M{I} + (m-2) \Vx\Vx^{\Tra} \bigr) 
  + m \bigl( \Msop{\VAxm}{\Vx} \bigr)
  \biggr] \\
  - \frac{m}{(\SBxm)^2} 
  \biggl[
  (m-1) \SAxm \TB\Vx^{m-2}
  + m \bigl( \Msop{\VAxm}{\VBxm} \bigr) \\
  + m \SAxm \bigl( \Msop{\Vx}{\VBxm} \bigr)
  \biggr].
\end{multline}
\end{theorem}

These complicated derivatives reduce for 
$\TB = \T{E}$. In that case, we have $\SBxm = 1$ and $\VBxm = \Vx$, so these equations become.
\begin{align*}
  \Vgx & = m \VAxm, &
  \MHx &= m(m-1) \MAxm.
\end{align*}
These match the derivatives of $f(x) = \SAxm$, as proved in
\cite{KoMa11}. Note that we have used the fact that $(m-1)\MExm =
\M{I} + (m-2) \Mxxt$ for all $\Vx \in \Sigma$.

We are considering the nonlinear program in \Eqn{opt} because there is
a correspondence between it and the generalized tensor eigenvalue
problem in \Eqn{gtep}. Note that the $\Vx$ in \Eqn{gtep} can be
arbitrarily rescaled.

\begin{theorem}
  Any pair $(\lambda,\Vx)$ is a solution to \Eqn{gtep} iff the scaled version with $\|\Vx\|=1$ is a KKT point of \Eqn{opt} with $\lambda$
  as the Lagrange multiplier.  
\end{theorem}
\begin{proof}
  First, assume $(\lambda,\Vx)$ is a solution to \Eqn{gtep}.
  Let constraint $\Vx \in \Sigma$ be expressed as $\Snxm = 1$. Then
  the Lagrangian is
  \begin{displaymath}
    \mathcal{L}(\Vx,\lambda) = 
    \Sfx - \lambda (\Snxm - 1).    
  \end{displaymath}
  Hence, using the derivatives in \App{derivatives}, we have
  \begin{equation}\label{eq:kkt}
    \nabla_{\Vx} \mathcal{L}(\Vx,\lambda) =
    \frac{m}{\SBxm} 
    \Biggl[
    \bigl( \SAxm \bigr) \; \Vx 
    \;+\; \VAxm
    \;-\; \biggl( \frac{\SAxm}{\SBxm} \biggr) \; \VBxm 
    \Biggr] - m \lambda \Vx = 0.
  \end{equation}
  So, $\Vx$ is a KKT point of \Eqn{opt} with Lagrange multiplier
  $\lambda$ as defined in \Eqn{lambda}.

  To prove the reverse, assume $\Vx$ is a KKT point of \Eqn{opt} with
  Lagrange multiplier $\lambda$.
  Then, \Eqn{kkt} must hold. If we multiply each term in \Eqn{kkt} by $\Vx$, then the third and fourth terms cancel out, and we conclude that $\lambda$ satisfies \Eqn{lambda}. Substituting that back into \Eqn{kkt}, we see that \Eqn{gtep} is satisfied. Hence, the claim.
\end{proof}

From the previous theorem, there is an equivalence between generalized tensor
eigenpairs and KKT points of \Eqn{opt}. Hence, solving \Eqn{opt}
yields eigenpairs. An eigenpair may correspond to a local maximum, a
local minimum, or a saddle point. For a given eigenpair
$(\lambda,\Vx)$ normalized so that $\Vx \in \Sigma$, we can categorize
it by considering the projected Hessian of the Lagrangian, i.e.,
\begin{equation}
  \M{C}(\lambda,\Vx) = \M{U}^{\Tra} 
  \bigl( \MHx - \lambda m \M{I} \bigr)
  \M{U} 
  \in \Real^{(n-1) \times (n-1)},
\end{equation}
where $\M{U} \in \Real^{n \times (n-1)}$ is an orthonormal basis for
$\Vx^{\perp}$. We can then say the following:
\begin{align*}
  \M{C}(\lambda,\Vx) & \text{ positive definite} & \Rightarrow & \text{ local minimum of \Eqn{opt}}, \\
  \M{C}(\lambda,\Vx) & \text{ negative definite} & \Rightarrow & \text{ local maximum of \Eqn{opt}}, \\
  \M{C}(\lambda,\Vx) & \text{ indefinite} & \Rightarrow & \text{ saddle point of \Eqn{opt}}. 
\end{align*}
The argument is very similar to that presented in \cite{KoMa11} and so
is omitted.  Optimization approaches cannot easily find saddle points,
but they can find local minima and maxima. We describe such an
optimization approach in the next section.

\section{Derivation of GEAP algorithm}
\label{sec:derivation}

We propose to use a property of convex functions of the sphere to
develop a monotonically convergent method. We consider an idea
originally from \cite{KoRe02,ReKo03}; see also \cite{KoMa11} for a
proof. We have modified  the theorem here to focus on its
\emph{local} applicability by considering just an open neighborhood of $\V{w}$ rather than all of $\Real^n$.
\begin{theorem}[Kofidis and Regalia \cite{KoRe02,ReKo03}]
  \label{thm:cvx}
  Let $f(\Vx)$ be a given function, and let $\V{w} \in \Sigma$ such that
  $\nabla f(\V{w}) \neq 0$. Let $\Omega$ be an open neighborhood of
  $\V{w}$, and assume $f$ is convex and continuously differentiable on
  $\Omega$.  Define $\V{v} = \nabla f(\V{w}) / \|\nabla f(\V{w})
  \|$. If $\V{v} \in \Omega$ and $\V{v} \ne \V{w}$, then
  $f(\V{v}) - f(\V{w}) > 0$.
\end{theorem}

\begin{corollary}
  \label{cor:cnv}
  Let $f(x)$ be a given function, and let $\V{w} \in \Sigma$ such that
  $\nabla f(\V{w}) \neq 0$. Let $\Omega$ be an open neighborhood of
  $\V{w}$, and assume $f$ is concave and continuously differentiable on
  $\Omega$.  Define $\V{v} = - \nabla f(\V{w}) / \|\nabla f(\V{w})
  \|$. If $\V{v} \in \Omega$ and $\V{v} \ne \V{w}$, then
  $f(\V{v}) - f(\V{w}) < 0$.
\end{corollary}

Hence, if $f$ is locally convex, then a simple algorithm, i.e.,
\begin{displaymath}
\V{x}_{\text{new}} = \Vgx / \| \Vgx \|.
\end{displaymath}
will yield ascent. Conversely, if $f$ is locally concave, we can
expect descent from $\V{x}_{\text{new}} = -\Vgx / \| \Vgx \|$. Unfortunately, the function $\Sfx$ in \Eqn{opt} may
not be convex or concave.

To fix this, we work with a shifted function,
\begin{equation}
  \Sfhatx = \Sfx + \alpha \Snxm.
\end{equation}
From \cite{KoMa11}, we have that for $\Vx \in \Sigma$, 
\begin{align}
  \label{eq:ghat}
  \Vghatx & \equiv \nabla \Sfhatx = \Vgx + \alpha m \Vx, \\
  \label{eq:hhat}
  \MHhatx & \equiv \nabla^2 \Sfhatx = \MHx + \alpha m \M{I} + \alpha m (m-2) \Vx\Vx^{\Tra}.
\end{align}
If we choose $\alpha$ appropriately, then we can ensure that $\MHhatx$
is positive or negative definite, ensuring that $\Sfhatx$ is locally convex or
concave.  In \cite{KoMa11} for the special case of $\TB=\T{E}$, we proposed choosing a single value for
$\alpha$ in SS-HOPM that ensured convexity on the entire sphere. But it is
difficult to choose a reasonable value in advance, and poor choices
lead to either very slow convergence or a complete lack of convergence.
In this work, we propose to choose $\alpha$ adaptively.

For an arbitrary matrix $n \times n$ symmetric matrix $\M{M}$, the
following notation denotes its eigenvalues: $\lambda_{\min}(\M{M}) = \lambda_1(\M{M}) \leq
\lambda_2(\M{M}) \leq \dots \leq \lambda_n(\M{M}) =
\lambda_{\max}(\M{M})$.

\begin{theorem}
  Assume $\Vx \in \Sigma$.  Let $\MHx$ and $\MHhatx$ be defined as in
  \Eqn{hess} and \Eqn{hhat}, respectively.  For $\alpha \geq 0$, the
  eigenvalues of $\MHhatx[]$ are bounded as
  \begin{equation}
  \label{eq:eigboundpos}
  \lambda_i(\MHx[]) + m \alpha 
  \leq \lambda_i(\MHhatx[])
  \leq \lambda_i(\MHx[]) + m \alpha + m(m-2) \alpha
  \end{equation}
  for $i=1,\dots\,n$. Likewise, for $\alpha \leq 0$, the eigenvalues
  of $\MHhatx[]$ are bounded as
  \begin{equation}
  \label{eq:eigboundneg}
  \lambda_i(\MHx[]) + m \alpha + m(m-2) \alpha
  \leq \lambda_i(\MHhatx[])
  \leq \lambda_i(\MHx[]) + m \alpha 
  \end{equation}
  for $i=1,\dots\,n$.
\end{theorem}
\begin{proof}
  The proof follows immediately from Weyl's inequality.
\end{proof}

In the convex case, our goal is to choose $\alpha$ so that $\MHhatx[]$
is positive semi-definite in a local neighborhood of the current
iterate, $\Vx$. By the smoothness of $\Sfhatx$ when $\Vx$ is away from
zero, we may argue that for every $\tau > 0$, there exists and $\delta
> 0$ such that $\MHhatx[]$ is positive semi-definite for all $\|\V{x}
- \V{x}_k\| \leq \delta$ whenever $\lambda_{\min}(\MHhatx[]) \geq
\tau$. In other words, $\tau$ is the threshold for positive
definiteness.
  
\begin{corollary}
  Assume $\Vx \in \Sigma$. Let $\tau > 0$.  If
  \begin{equation}
  \alpha = \max\set{ 0, (\tau - \lambda_{\min}(\MHx[]))/m},
  \end{equation}
  then $\lambda_{\min} (\MHhatx[]) \geq \tau$.
\end{corollary}

In the concave case, our goal is to choose $\alpha$ so that $\MHhatx[]$ is negative semi-definite
in a local neighborhood of the current iterate.
\begin{corollary}
  Assume $\Vx \in \Sigma$. Let $\tau > 0$.  If
  \begin{equation}
  \alpha = \min\set{ 0, -(\tau + \lambda_{\max}(\MHx[]))/m}
         = - \max \set{0, \tau - \lambda_{\min}(-\MHx[]))/m},
  \end{equation}
  then $\lambda_{\max} (\MHhatx[]) \leq -\tau$.
\end{corollary}

From \Thm{cvx},  if $\alpha$ is selected to make $\Sfhatx$ locally convex, we have
\begin{displaymath}
  \Vx_+ = \Vghatx / \| \Vghatx \|
  \quad\Rightarrow\quad
  \Sfhatx[(\Vx_+)] > \Sfhatx
  \quad\Rightarrow\quad
  \Sfx[(\Vx_+)] > \Sfx
\end{displaymath}
so long as $\Vx_+ \in \Omega$, the convex neighborhood of $\Vx$.
Even though we are adaptively changing $\alpha$, 
we see increase in the original function at each step.
A similar argument applies in the concave case, 
with the function decreasing at each step.

The potential problem with this approach is that it may be the case that $\Vx_+
\not\in \Omega$. 
If that happens, 
we may observe that the function values (i.e., $\lambda_k$) are 
not increasing (or decreasing) monotonically
as expected. To fix this, we make a more conservative choice for $\tau$
(at least temporarily), which
will in turn enforce a more conservative choice of $\alpha$.
If $\tau$ is large enough, then we will satisfy the lower bound on $\alpha$ that guarantees convergence for the shifted algorithm (this is proven for Z-eigenvalue in \cite{KoMa11}; the proof for the general problem is similar and so omitted).
Thus far in our experiments, such contingencies have not been necessary, 
so we have not included the details in the algorithm.

\subsection{GEAP Algorithm}
\label{sec:geapalg}
The full algorithm is presented in \Alg{geap}.

\begin{algorithm}
  \caption{Generalized Eigenpair Adaptive Power (GEAP) Method}   
  \label{alg:geap}  
  Given tensors $\TA \in \Smn$ and $\TB \in \Smn_+$ and an initial guess
  $\V[\hat]{x}_0$. Let $\beta=1$ if we want to find local maxima (and the function
  is convex); otherwise, let $\beta=-1$, indicating that we are seeking local
  minima (and the function is concave). Let $\tau$ be the tolerance on being positive/negative definite.
  \begin{algorithmic}[1]
    \State $\Vx_0 \gets \V[\hat]{x}_0 / \| \V[\hat]{x}_0 \|$
    \For{$k=0,1,\dots$}
    \State Precompute $\MAxm_k$, $\MBxm_k$, $\VAxm_k$, $\VBxm_k$, $\SAxm_k$, $\SBxm_k$
    \State $\lambda_k \gets {\SAxm_k}/{\SBxm_k}$
    \State $\M{H}_k \gets \MHx[(\Vx_k)]$    
    \State $\alpha_k \gets \beta \max \{0, (\tau - \lambda_{\min}(\beta \M{H}_k))/m\}$
    \State $\V[\hat]{x}_{k+1} \gets \beta \bigl( \VAxm_k - \lambda_k \VBxm_k + (\alpha_k + \lambda_k) \SBxm_k \Vx_k \bigr)$
    \State $\Vx_{k+1} = \V[\hat]{x}_{k+1} / \| \V[\hat]{x}_{k+1} \|$
    \EndFor
  \end{algorithmic}
\end{algorithm}

The cost per iteration of \Alg{geap} is as follows. Assuming $\TA$ and $\TB$ are dense, the dominant cost is computing products with these tensors. 
Computing the Hessian requires six products:
$\MAxm, \VAxm, \SAxm, \MBxm, \VBxm, \SBxm$.
Recall that $\MAxm$ is given by
\begin{displaymath}
  (\MAxm)_{i_1i_2} = \sum_{i_3=1}^n \cdots \sum_{i_m=1}^n 
  a_{i_1 i_2 \dots i_m} x_{i_3} \cdots x_{i_m}
  \qtext{for all} i_1,i_2 = 1,\dots,n.
\end{displaymath}
The cost is $(m-2)\cdot n^{m-2}$ multiplies and $n^{m-2}$ additions per each of $n^2$ entries; therefore, the total cost is $(m-1)n^m$ operations. 
Exploiting symmetry yields reduces the cost to $O(n^m/m!)$ \cite{ScLoVaKoXX}.
We can compute
\begin{displaymath}
  \VAxm = (\MAxm) \Vx \qtext{and} \SAxm = (\VAxm)^{\Tra} \Vx,
\end{displaymath}
for an additional cost of $2n^2$ and $2n$ operations, respectively. (These can also be computed directly, but at a cost of $(m-1)n^m$ operations each.)
These values have to be computed for both $\TA$ and $\TB$ at every iteration for a total cost (ignoring symmetry) of $2(m-1)n^m + 4n^2 + 4n$.
Once these six products are computed, the cost for computing $\M{H}(\Vx)$ is a series of matrix operations, for a total cost of $20n^2$.
The cost of computing the eigenvalues of the symmetric matrix $\M{H}(\Vx)$ is $(4/3) n^3$, which is less than the cost of the products.
Updating $\Vx$ requires a 5 vector operations at a cost of $n$ operations each. Hence, cost of the method is dominated by the computation of $\MAxm$ and $\MBxm$, at an expense of $O(n^m/m!)$.

\subsection{Specialization of GEAP to Z-eigenpairs}
In \Alg{zeap}, we show the specialization of the method to the Z-eigenvalue
problem. This is the same as SS-HOPM, except for the adaptive shift. Note that
unlike \Alg{geap}, this algorithm can be used even when $m$ is odd.
The cost per iteration of \Alg{zeap} is the same order as for \Alg{geap}, but it does not need to do any computations with $\T{B}$.
\begin{algorithm}
  \caption{Z-Eigenpair Adaptive Power Method}   
  \label{alg:zeap}  
  Given tensor $\TA \in \Smn$ and an initial guess
  $\V[\hat]{x}_0$. Let $\beta=1$ if we want to find local maxima (and the function
  is convex); otherwise, let $\beta=-1$, indicating that we are seeking local
  minima (and the function is concave). Let $\tau$ be the tolerance on being positive/negative definite.
  \begin{algorithmic}[1]
    \State $\Vx_0 \gets \V[\hat]{x}_0 / \| \V[\hat]{x}_0 \|$
    \For{$k=0,1,\dots$}
    \State Precompute $\MAxm_k$, $\VAxm_k$, $\SAxm_k$
    \State $\lambda_k \gets {\SAxm_k}$
    \State $\M{H}_k \gets m(m-1)\MAxm_k$    
    \State $\alpha_k \gets \beta \max \{0, (\tau - \lambda_{\min}(\beta \M{H}_k))/m\}$
    \State $\V[\hat]{x}_{k+1} \gets \beta \bigl( \VAxm_k + \alpha_k \Vx_k \bigr)$
    \State $\Vx_{k+1} = \V[\hat]{x}_{k+1} / \| \V[\hat]{x}_{k+1} \|$
    \EndFor
  \end{algorithmic}
\end{algorithm}

\section{Numerical experiments}
\label{sec:experiments}

All numerical tests were done using MATLAB Version R2012b and the Tensor Toolbox Version 2.5 \cite{TTB}. The experiments were performed a laptop computer with an Intel Dual-Core i7-3667UCPU (2GHz) and 8GB of RAM.

In all numerical experiments, we used the following settings.
We set $\tau=10^{-6}$, where $\tau$ is the tolerance on being positive or negative definite.
We consider the iterates to be converged once $|\lambda_{k+1} - \lambda_k| \leq 10^{-15}$.
The maximum iterations is 500.

\subsection{Comparison to SS-HOPM for computing Z-eigenpairs}

The following example is originally from \cite{KoRe02} and was used in evaluating the SS-HOPM algorithm in \cite{KoMa11}. Our goal is to compute the Z-eigenpairs \Eqn{ztep} using the Z-Eigenpair Adaptive Power Method in \Alg{zeap} and show that it is faster than SS-HOPM~\cite{KoMa11} using a fixed value for the shift.

\begin{exampleplusname}{{Kofidis and Regalia \cite{KoRe02}}}%
  \label{ex:KoRe02ex1}%
 Our objective is to compute the Z-eigenpairs. 
 Let $\TA \in \Smn[4,3]$ be the symmetric tensor given by Kofidis and
 Regalia \cite[Example 1]{KoRe02} whose  entries are specified in \App{entries}
 (\Fig{KoRe02}).
 Since we are computing Z-eigenpairs, we have $\TB = \T{E}$.
  A complete list of the 11 Z-eigenpairs is provided in \App{eigs} (\Tab{KoRe02ex1_ep});
  there are three maxima and three minima.
\end{exampleplusname}

A comparison of the fixed and adaptive shift results are provided in \Tab{KoRe02ex1}. 
There are six different experiments looking at maxima ($\beta=1$) and minima ($\beta=-1$) and different shifts ($\alpha=2, 10, \text{adaptive}$) in \Alg{zeap}.
Note that using a fixed shift means that \Alg{zeap} is equivalent to SS-HOPM and no adaptive update of the shift is performed in Step 5. 

We used 100 random starting guesses, each entry selected uniformly at random
from the interval $[-1,1]$; the same set of random starts was used for each set
of experiments.
For each eigenpair, the table lists the number of occurrences in the 100
experiments, the median number of iterations until convergence, 
the number of runs that violated monotonicity, the average error and standard deviation in the final result, and the average run time and standard deviation.
The error is computed as $\| \VAxm - \lambda \Vx \|_2$. 
The two monotinicity violations were both extremely small, i.e., O($10^{-12}$).
These violations indicate that a step went outside the region of local convexity. 

\begin{table}[pthb]
  \setlength{\tabcolsep}{3pt}
  \centering
  \footnotesize
  \caption{Different shifts to calculate Z-eigenpairs for $\TA\in \Smn[2,3]$ from \Ex{KoRe02ex1}.}
  \label{tab:KoRe02ex1}
\begin{ResultsSubTable}{$\alpha$ adaptive, $\beta=1$}{3}
  53 &  0.8893 &  0.6672 &  0.2471 & -0.7027 &   30 & \multicolumn{2}{c|}{--} & 9e-09 & 3e-09 & 0.05 & 0.02 \\ \hline 
  29 &  0.8169 &  0.8412 & -0.2635 &  0.4722 &   34 & \multicolumn{2}{c|}{--} & 1e-08 & 3e-09 & 0.04 & 0.01 \\ \hline 
  18 &  0.3633 &  0.2676 &  0.6447 &  0.7160 &   26 & \multicolumn{2}{c|}{--} & 7e-09 & 2e-09 & 0.03 & 0.00 \\ \hline 
\end{ResultsSubTable}

\begin{ResultsSubTable}{$\alpha=2$, $\beta=1$}{3}
  53 &  0.8893 &  0.6672 &  0.2471 & -0.7027 &   49 &   1 & 1e-15 & 2e-08 & 3e-09 & 0.07 & 0.03 \\ \hline 
  29 &  0.8169 &  0.8412 & -0.2635 &  0.4722 &   45 & \multicolumn{2}{c|}{--} & 2e-08 & 3e-09 & 0.04 & 0.00 \\ \hline 
  18 &  0.3633 &  0.2676 &  0.6447 &  0.7160 &   57 & \multicolumn{2}{c|}{--} & 2e-08 & 2e-09 & 0.06 & 0.00 \\ \hline 
\end{ResultsSubTable}

\begin{ResultsSubTable}{$\alpha=10$, $\beta=1$}{3}
  48 &  0.8893 &  0.6672 &  0.2471 & -0.7027 &  192 & \multicolumn{2}{c|}{--} & 5e-08 & 6e-09 & 0.24 & 0.12 \\ \hline 
  29 &  0.8169 &  0.8412 & -0.2635 &  0.4722 &  185 & \multicolumn{2}{c|}{--} & 5e-08 & 5e-09 & 0.17 & 0.02 \\ \hline 
  18 &  0.3633 &  0.2676 &  0.6447 &  0.7160 &  261 & \multicolumn{2}{c|}{--} & 5e-08 & 2e-09 & 0.24 & 0.02 \\ \hline 
   5 & \multicolumn{5}{|c|}{\emph{Failed to converge in 500 iterations}} & \multicolumn{2}{c|}{--} & \multicolumn{2}{c|}{N/A} & 0.43 & 0.01 \\ \hline 
\end{ResultsSubTable}

\begin{ResultsSubTable}{$\alpha={\rm adaptive}$, $\beta=-1$}{3}
  22 & -0.0451 &  0.7797 &  0.6135 &  0.1250 &   18 & \multicolumn{2}{c|}{--} & 4e-09 & 2e-09 & 0.02 & 0.00 \\ \hline 
  37 & -0.5629 &  0.1762 & -0.1796 &  0.9678 &   17 & \multicolumn{2}{c|}{--} & 6e-09 & 2e-09 & 0.02 & 0.00 \\ \hline 
  41 & -1.0954 &  0.5915 & -0.7467 & -0.3043 &   17 & \multicolumn{2}{c|}{--} & 6e-09 & 3e-09 & 0.02 & 0.01 \\ \hline 
\end{ResultsSubTable}

\begin{ResultsSubTable}{$\alpha=-2$, $\beta=-1$}{3}
  22 & -0.0451 &  0.7797 &  0.6135 &  0.1250 &   34 & \multicolumn{2}{c|}{--} & 1e-08 & 2e-09 & 0.03 & 0.00 \\ \hline 
  37 & -0.5629 &  0.1762 & -0.1796 &  0.9678 &   20 & \multicolumn{2}{c|}{--} & 7e-09 & 2e-09 & 0.02 & 0.00 \\ \hline 
  41 & -1.0954 &  0.5915 & -0.7467 & -0.3043 &   21 &   1 & 1e-15 & 8e-09 & 4e-09 & 0.02 & 0.00 \\ \hline 
\end{ResultsSubTable}

\begin{ResultsSubTable}{$\alpha=-10$, $\beta=-1$}{3}
  22 & -0.0451 &  0.7797 &  0.6135 &  0.1250 &  186 & \multicolumn{2}{c|}{--} & 5e-08 & 2e-09 & 0.16 & 0.01 \\ \hline 
  37 & -0.5629 &  0.1762 & -0.1796 &  0.9678 &  103 & \multicolumn{2}{c|}{--} & 4e-08 & 4e-09 & 0.09 & 0.01 \\ \hline 
  41 & -1.0954 &  0.5915 & -0.7467 & -0.3043 &   94 & \multicolumn{2}{c|}{--} & 4e-08 & 6e-09 & 0.09 & 0.01 \\ \hline 
\end{ResultsSubTable}

\end{table}

The first three experiments use $\beta=1$ to look for local
maxima. The first experiment varies $\alpha$, the second uses
$\alpha=2$ (as in \cite{KoMa11}), and the third uses $\alpha=10$.  All
three variations find all three local maxima. The results for
$\alpha=2$ and the adaptive method are nearly identical --- they find
the same local maxima with the same frequency. The difference is that
$\alpha=2$ uses more iterations than the adaptive shift. Choosing $\alpha=10$ is
similar, except now five of the runs do not converge within the
allotted 500 iterations.  There was no breakdown in monotonicity, and
these runs would converge eventually. If the shift is too small (e.g.,
$\alpha=0$), then some or all of the runs may fail to
converge~\cite{KoMa11}.

The last three experiments use $\beta=-1$ to find local minima. Again,
we vary $\alpha$ using an adaptive choice along with $\alpha=-2$ and
$\alpha=-10$. The adaptive method requires the fewest number of
iterations. Each experiments finds all three local minima with the
exact same frequencies.

To compare the convergence in terms of the number of iterations, \Fig{KoRe02ex1} shows sample results for one run for computing Z-eigenpairs of
$\TA$ from \Ex{KoRe02ex1}. The left hand plot shows the selected shift values
at each iteration. The right hand plot shows the convergence of the eigenvalue.
The adaptive shift is the fastest to converge.

\begin{figure}[bhtp]
  \centering
  \subfloat[Adapative Shift Selection]{\includegraphics[width=0.45\textwidth]{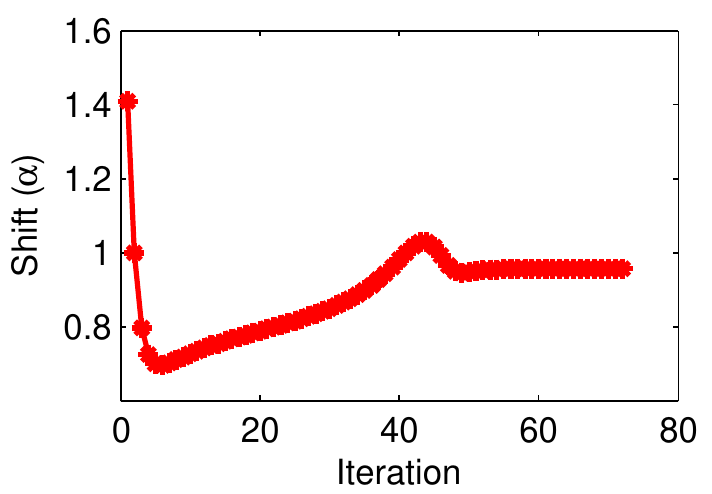}}~~
  \subfloat[Convergence of eigenvalues]{\includegraphics[width=0.45\textwidth]{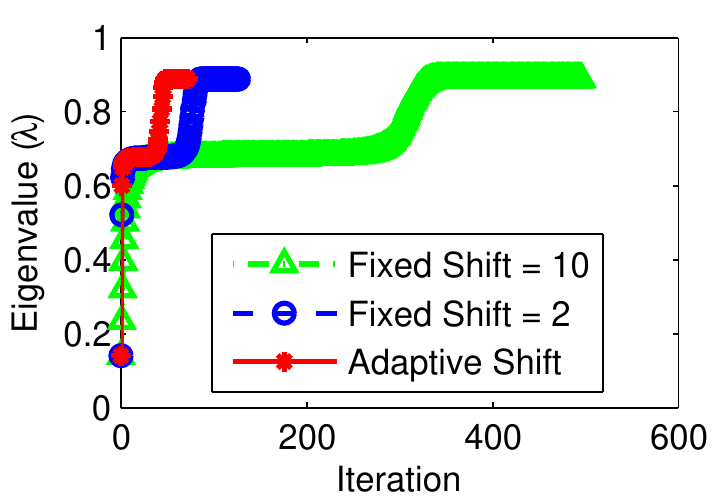}}
  \caption{GEAP sample results for $\TA\in \mathbb{S}^{[4,3]}$ from
    \Ex{KoRe02ex1} with $\beta =1$ and starting point $\Vx_0$ = [ 0.0417 -0.5618 0.6848 ].}
  \label{fig:KoRe02ex1}
\end{figure}

\subsection{Numerical results for H-eigenpairs}
\label{sec:heig}
Here we demonstrate that the GEAP method in \Alg{geap} calculates
H-eigenpairs \Eqn{htep} with an appropriate choice for $\T{B}$.

\begin{example}%
  \label{ex:heig}
  We generate a random symmetric tensor $\TA \in \Smn[6,4]$ as follows:
  we select
  random entries from $[-1,1]$, symmetrize the result, and round to 
  four decimal places. The tensor
  entries are specified in \App{entries}
 (\Fig{random_adef}).
 Since we are computing H-eigenpairs, we specify $\T{B}$ as
 $b_{i_1i_2\dots i_m} = \delta_{i_1i_2\dots i_m}$.
  A complete list of the H-eigenpairs is provided in \App{eigs} (\Tab{heig_ep});
  there are five maxima and five minima.
\end{example}

A summary of the results are provided in \Tab{heig}. 
There are two different experiments looking at maxima ($\beta=1$) and minima ($\beta=-1$).
We used 1000 random starting guesses, each entry selected uniformly at random
from the interval $[-1,1]$; the same set of random starts was used for each 
experiment.
The columns are the same as for \Tab{KoRe02ex1}.
The error is computed as $\| \VAxm - \lambda \Vx^{[m-1]} \|_2$. 

\begin{table}[htpb]
  \setlength{\tabcolsep}{2.5pt}
  \centering\footnotesize
  \caption{H-eigenpairs for $\T{A}\in\Smn[6,4]$ from \Ex{heig}}
  \label{tab:heig}
\begin{ResultsSubTable}{$\beta=1$}{4}
 211 & 14.6941 &  0.5426 & -0.4853 &  0.4760 &  0.4936 &   28 & 124 & 2e-15 & 2e-09 & 2e-09 & 0.08 & 0.02 \\ \hline 
 144 &  9.6386 &  0.5342 & -0.5601 &  0.5466 & -0.3197 &  110 &  53 & 4e-15 & 9e-09 & 3e-09 & 0.27 & 0.04 \\ \hline 
 338 &  8.7371 &  0.4837 &  0.5502 &  0.6671 & -0.1354 &  100 & 143 & 4e-15 & 1e-08 & 5e-09 & 0.26 & 0.06 \\ \hline 
 169 &  5.8493 &  0.6528 &  0.5607 & -0.0627 & -0.5055 &   54 &  19 & 2e-15 & 8e-09 & 3e-09 & 0.13 & 0.02 \\ \hline 
 138 &  4.8422 &  0.5895 & -0.2640 & -0.4728 &  0.5994 &   66 &  13 & 8e-01 & 6e-09 & 1e-09 & 0.16 & 0.01 \\ \hline 
\end{ResultsSubTable}

\begin{ResultsSubTable}{$\beta=-1$}{4}
 130 & -2.9314 &  0.3161 &  0.5173 &  0.4528 & -0.6537 &   76 &   3 & 2e-15 & 7e-09 & 1e-09 & 0.18 & 0.02 \\ \hline 
 149 & -3.7179 &  0.6843 &  0.5519 &  0.3136 &  0.3589 &   59 &  10 & 1e-15 & 7e-09 & 2e-09 & 0.15 & 0.02 \\ \hline 
 152 & -4.1781 &  0.4397 &  0.5139 & -0.5444 &  0.4962 &   99 &   7 & 2e-15 & 5e-09 & 1e-09 & 0.23 & 0.03 \\ \hline 
 224 & -8.3200 &  0.5970 & -0.5816 & -0.4740 & -0.2842 &   65 &  73 & 2e-15 & 8e-09 & 3e-09 & 0.16 & 0.02 \\ \hline 
 345 & -10.7440 &  0.4664 &  0.4153 & -0.5880 & -0.5140 &   47 & 181 & 2e-15 & 4e-09 & 3e-09 & 0.12 & 0.03 \\ \hline 
\end{ResultsSubTable}
\end{table}

The first experiment uses $\beta=1$ to look for local maxima. 
We find all five local maxima. 
There are several monotonicity violations, including at least one for $\lambda=4.8422$ that is relatively large.
These violations indicate that a step went outside the region of local convexity. 
Nevertheless, in all cases the algorithm is able to recover and converge, as can be seen from the small error. 
In general, these monotonicity violations do not cause the algorithm to fail. However,
such violations can be avoided by increasing $\tau$, the tolerance on the definiteness of the Hessian matrix. Once $\tau$ is large enough, the shift will be so great that the function will be convex over the entire unit sphere. The downside of choosing a large value for $\tau$ (and the shift $\alpha_k$) is that convergence will be slow. 

The second experiment uses $\beta = -1$ to look for local minima.
We find all five local minima. 
There are several monotinicity violations, but they are all small.

\subsection{Numerical results for D-eigenpairs}

Next we consider a different type of tensor eigenapair that also
conforms to the generalized tensor eigenpair framework.

\begin{exampleplusname}{{D-eigenpairs \cite{QiWaWu08}}}
  \label{ex:deig}%
  Qi, Wang, and Wu \cite{QiWaWu08} propose D-eigenpairs for diffusion kurtosis  
  imaging (DKI). The tensors $\TA,\TB \in \Smn[4,3]$ are specified in \App{entries} (\Fig{deiga} and \Fig{deigb}, respecitively). We consider this example here since it can be expressed as a
  generalized tensor eigenproblem.\footnote{%
    Note that only four digits of precision for $\TA$ and $\M{D}$ are provided
    in \cite{QiWaWu08}. We were unable to validate the solutions provided in
    the original paper. It is not clear if this is to a lack of precision or a
    typo in paper. Here, the problem is rescaled as well: $\M{D}$ is multiplied
    by $10^2$, $\lambda$ is divided by $10^4$, and $\Vx$ is divided by $10$.}
 There are a total of 13 distinct real-valued D-eigenpairs, computed
  by solving the polynomial equations using Mathematica and listed in
  \App{eigs} (\Tab{deig}); there are four maxima and three minima.
\end{exampleplusname}

\Tab{deig_results} shows the eigenpairs calculated by \Alg{geap}. 
The error is computed as $\| \VAxm - \lambda \VBxm \|_2$. 
With 100 random starts, we find the four local maxima with $\beta=1$. Likewise, with 100 random starts, we find the three local minima with $\beta=-1$. There are no violations to monotonicity.

\begin{table}[htbp]
  \setlength{\tabcolsep}{2.5pt}
  \centering
  \footnotesize
  \caption{D-eigenpairs for $\TA\in \Smn[2,3]$ and $\M{D} \in \Smn[2,3]$ from \Ex{deig}}
  \label{tab:deig_results}
  \begin{ResultsSubTable}{$\beta=1$}{3}
  31 &  0.5356 &  0.9227 & -0.1560 & -0.3526 &   39 & \multicolumn{2}{c|}{--} & 4e-08 & 5e-09 & 0.08 & 0.01 \\ \hline 
  19 &  0.4359 &  0.5388 &  0.8342 & -0.1179 &   48 & \multicolumn{2}{c|}{--} & 3e-08 & 4e-09 & 0.09 & 0.02 \\ \hline 
  25 &  0.2514 &  0.3564 & -0.8002 &  0.4823 &   67 & \multicolumn{2}{c|}{--} & 4e-08 & 3e-09 & 0.13 & 0.01 \\ \hline 
  25 &  0.2219 &  0.2184 &  0.3463 &  0.9124 &   34 & \multicolumn{2}{c|}{--} & 6e-08 & 8e-09 & 0.07 & 0.01 \\ \hline 
  \end{ResultsSubTable}
  \begin{ResultsSubTable}{$\beta=-1$}{3}
  39 & -0.0074 &  0.3669 &  0.5346 & -0.7613 &   13 & \multicolumn{2}{c|}{--} & 1e-08 & 4e-09 & 0.03 & 0.01 \\ \hline 
  37 & -0.1242 &  0.9439 &  0.1022 &  0.3141 &   51 & \multicolumn{2}{c|}{--} & 5e-08 & 5e-09 & 0.10 & 0.01 \\ \hline 
  24 & -0.3313 &  0.2810 & -0.9420 & -0.1837 &   27 & \multicolumn{2}{c|}{--} & 2e-08 & 4e-09 & 0.06 & 0.01 \\ \hline 
  \end{ResultsSubTable}
\end{table}

\subsection{Generalized eigenpairs for randomly generated $\TA$ and $\TB$}

Here we consider a randomly generated problem. 
We use the randomly generated $\TA \in \Smn[6,4]$ described in \Sec{heig}.
However, we need a method to generate a positive definite $\TB$. We use the notation $\TB
= (\T{E},\M{S}, \dots, \M{S}) \in \Smn$ to denote tensor-matrix multiplication in which the tensor $\T{E}$ is multiplied by a matrix $\M{S} \in \Real^{n
    \times n}$ in every mode, i.e.,
  \begin{displaymath}
    b_{i_1 \cdots i_m} = \sum_{j_1 = 1}^n \cdots \sum_{j_m=1}^n e_{j_1 \cdots j_m} s_{i_1 j_1} \cdots s_{i_m j_m}.
  \end{displaymath}
\begin{theorem}\label{thm:posdef}
  Let $\M{S} \in \Real^{n \times n}$ be symmetric.  For $m$ even, define $\TB
  \in \Smn$ as $\TB = (\T{E},\M{S}, \dots, \M{S})$.  If $(\mu,\Vx)$ is
  a real-valued eigenpair of $\M{S}$ and $\|\Vx\|=1$, then
  $(\lambda,\Vx)$ is a Z-eigenpair of $\T{B}$ with $\lambda =
  \mu^m$. Furthermore, $\TB\Vy^m \geq \min_i (\mu_i)^m$ for any $\Vy \in \Real^n$ with $\|\Vy\|=1$.
\end{theorem}
\begin{proof}
  Let $(\mu,\Vx)$ be an eigenpair of $\M{S}$ such that $\|\Vx\| = 1$. Noting that $\M{S}^{\Tra} = \M{S}$, we have
$
  \TB\Vx^{m-1} 
  = \M{S} \T{E}(\M{S}\Vx)^{m-1} 
  = \M{S} \T{E} (\mu \Vx)^{m-1}
  =  \mu^{m-1} \M{S} \T{E} \Vx^{m-1}
  = \mu^{m-1} \M{S} \Vx
  = \mu^m \Vx
  .
$
To prove the lower bound, 
let $\Vy \in \Real^n$ with $\|\Vy\|=1$. We can write $\Vy$ as a linear combination of the eigenvectors of $\M{S}$, i.e., $\Vy = \sum_i \nu_i \Vx_i$ and $\sum \nu_i^2 = 1$. Then
\begin{align*}
  \TB \Vy^m &= \Vy^{\Tra} \TB \Vy^{(m-1)} 
  = \Vy^{\Tra} \M{S} \T{E} (\M{S} \Vy)^{(m-1)} 
  = \| \M{S} \Vy \|^{m-1} (\M{S}\Vy)^{\Tra} (\M{S} \Vy) = \| \M{S} \Vy\|^m \\
  &= \| \sum_i \mu_i \nu_i \Vx_i \|^m  
  = \left( \sqrt{ \sum_i \mu_i^2 \nu_i^2 }\right)^m
  \geq \left( \sqrt{ \min_i \mu_i^2 } \right)^m = \min_i (\mu_i)^m
\end{align*}
Hence, the claim. 
\end{proof}

\begin{exampleplusname}{Random}\label{ex:random}
  We use the same randomly-generated symmetric tensor $\TA \in
  \Smn[6,4]$ as for \Ex{heig} and listed in \App{entries} (\Fig{random_adef}).
  To generate a random positive definite symmetric tensor $\TB \in
  \Smn_+ = \Smn[6,4]_+$, we use \Thm{posdef}.  (Note that this approach samples a convenient subset
	of $\Smn_+$ and does not draw from the entire space.)  We compute a matrix $\M{S} = \M{U}
  \M{D} \M{U}^{\Tra} \in \Real^{4 \times 4}$, where $\M{U} \in \Real^{4
    \times 4}$ is a random orthonormal matrix and $\M{D} \in \Real^{4
    \times 4}$ is a diagonal matrix with entries selected uniformly at
  random from $[-1,-\gamma] \cup [\gamma,1]$ with $\gamma = \sqrt[m]{0.1} = \sqrt[6]{0.1}$. We let
  $\TB = (\T{E}, \M{S},\dots,\M{S})$, so
	that $\TB$ has all its Z-eigenvalues in $[0.1,1]$ and is positive definite.  In this case, the randomly
  selected diagonal for $\M{D}$ is $\set{-0.8620, 0.8419, 0.7979,
    0.6948}$.
  The random $\T{B}$ is then rounded to four decimal places, and the entries are  given in \App{entries} (\Fig{random_bdef}).  Its minimum Z-eigenvalue (computed
  by GEAP) is $0.1125 = 0.6948^6$, as expected.
  There are a total of 26 real-valued $\TB$-eigenpairs of $\TA$,
  listed in \App{eigs} (\Tab{random_eigs}). There are three maxima and four minima.
\end{exampleplusname}

\Tab{random_results} shows the generalized eigenpairs calculated by \Alg{geap}. 
The error is computed as $\| \VAxm - \lambda \VBxm \|_2$. 
With 1000 random starts, we find the three local maxima with $\beta=1$. Likewise, with 1000 random starts, we find the four local minima with $\beta=-1$. There are only small violations to monotonicity; the maximum of any violation was O($10^{-14}$).

\begin{table}[htbp]
  \setlength{\tabcolsep}{2.5pt}
  \centering\footnotesize
  \caption{Generalized eigenpairs for $\TA,\TB\in\Smn[6,4]$ from \Ex{random}}
  \label{tab:random_results}

 \begin{ResultsSubTable}{$\beta=1$}{4}
683 & 11.3476 &  0.4064 &  0.2313 &  0.8810 &  0.0716 &   59 & 420 & 4e-15 & 5e-09 & 4e-09 & 0.14 & 0.03 \\ \hline 
 128 &  3.7394 &  0.2185 & -0.9142 &  0.2197 & -0.2613 &  140 &  11 & 2e-15 & 1e-08 & 3e-09 & 0.30 & 0.04 \\ \hline 
 189 &  2.9979 &  0.8224 &  0.4083 & -0.0174 & -0.3958 &   23 &   9 & 1e-15 & 3e-09 & 1e-09 & 0.06 & 0.01 \\ \hline 
 \end{ResultsSubTable}

 \begin{ResultsSubTable}{$\beta=-1$}{4}
 151 & -1.1507 &  0.1935 &  0.5444 &  0.2991 & -0.7594 &   88 & \multicolumn{2}{c|}{--} & 8e-09 & 8e-10 & 0.19 & 0.02 \\ \hline 
 226 & -3.2777 &  0.6888 & -0.6272 & -0.2914 & -0.2174 &   33 &  14 & 1e-15 & 6e-09 & 2e-09 & 0.08 & 0.01 \\ \hline 
 140 & -3.5998 &  0.7899 &  0.4554 &  0.2814 &  0.2991 &   22 &  21 & 1e-15 & 2e-09 & 1e-09 & 0.05 & 0.01 \\ \hline 
 483 & -6.3985 &  0.0733 &  0.1345 &  0.3877 &  0.9090 &   82 &  73 & 2e-15 & 9e-09 & 3e-09 & 0.17 & 0.03 \\ \hline 
 \end{ResultsSubTable}
\end{table}

\section{Related work}
\label{sec:related}

Like its predecessor SS-HOPM \cite{KoMa11}, the GEAP method has the desirable
qualities of guaranteed convergence and simple
implementation. Additionally, the adaptive choice of $\alpha$ in GEAP
(as opposed to SS-HOPM) means that there are no parameters for the
user to specify.

Also like SS-HOPM, the GEAP method can only converge to local maxima
and minima of \Eqn{opt} and so will miss any saddle point solutions.
Nevertheless, the largest and smallest magnitude eigenvalues can
always be discovered by GEAP since they will not be saddle points.

\subsection{Numerical Optimization Approaches}
An alternative to GEAP is to solve \Eqn{gtep} or \Eqn{opt} using a
numerical nonlinear, homotopy, or optimization approach. The advantage
of GEAP is that is guarantees decrease at each iteration
\emph{without} any globalization techniques (like line search or trust
region) and is generally as cheap or cheaper per iteration than any
competing numerical method. The disadvantage is that the rate of
convergence of GEAP is only linear, as opposed to quadratic for, say,
Newton's method.

Han \cite{Ha12} proposed an unconstrained variations principle for finding generalized eigenpairs. In the general case, the function to be optimized is
\begin{equation}
  \label{eq:han}
  f(\Vx) = \frac{(\SBxm)^2}{2m} - \beta \frac{\SAxm}{m}.
\end{equation}
The $\beta$ has the same meaning as for GEAP: choosing $\beta=1$ finds
local maxima and $\beta=-1$ finds local minima.  For comparison, the final solution is rescaled as $\Vx = \Vx / \sqrt[m]{\SBxm}$, and then we calculate $\lambda = \SAxm$ (since $\SBxm = 1$).

The computational experiment settings are the same as specified in
\Sec{experiments}. Han used the MATLAB Optimization
Toolbox, and we use Version 2.6.1. Folliwng Han, we use the \texttt{fminunc}
function and the default settings from calling
\texttt{optimset('fminunc')} except that we explicitly specify
\begin{itemize}
\item \texttt{GradObj:on}
\item \texttt{LargeScale:off}
\item \texttt{TolX:1e-10}
\item \texttt{TolFun:1e-8}
\item \texttt{MaxIter:10000}
\item \texttt{Display:off}
\end{itemize}
This means that the toolbox uses a quasi-Newton method with a line
search which should have superlinear convergence.

The results of Han's method for \Ex{random} are shown in \Tab{han_random_results}. 
For each eigenpair, the table lists the number of occurrences in the 1000
experiments, the median number of function evaluations (fevals) until convergence, 
the average error and standard deviation in the final result, and the average run time and standard deviation.
The error is computed as $\| \VAxm - \lambda \VBxm \|_2$; both methods achieved comparable errors.

\begin{table}[htbp]
  \setlength{\tabcolsep}{2.5pt}
  \centering\footnotesize
  \caption{Generalized eigenpairs from Han's method for $\TA,\TB\in\Smn[6,4]$ from \Ex{random}}
  \label{tab:han_random_results}

 \begin{HanResultsSubTable}{$\beta=1$}{4}
 718 & 11.3476 &  0.5544 &  0.3155 &  1.2018 &  0.0977 &   45 & 1e-08 & 2e-08 & 0.17 & 0.06 \\ \hline 
 134 &  3.7394 &  0.2642 & -1.1056 &  0.2657 & -0.3160 &   31 & 4e-09 & 7e-09 & 0.12 & 0.05 \\ \hline 
 144 &  2.9979 &  1.0008 &  0.4969 & -0.0212 & -0.4817 &   31 & 4e-09 & 5e-09 & 0.12 & 0.05 \\ \hline 
   4 & \multicolumn{8}{c|}{\emph{--- Failed to converge ---}} & 0.21 & 0.10 \\ \hline 
 \end{HanResultsSubTable}

 \begin{HanResultsSubTable}{$\beta=-1$}{4}
  72 & -1.1507 &  0.2291 &  0.6444 &  0.3540 & -0.8990 &   34 & 9e-09 & 3e-08 & 0.14 & 0.06 \\ \hline 
 150 & -3.2777 &  0.8349 & -0.7603 & -0.3532 & -0.2635 &   33 & 5e-09 & 7e-09 & 0.14 & 0.07 \\ \hline 
 148 & -3.5998 &  1.0486 &  0.6046 &  0.3736 &  0.3971 &   41 & 6e-09 & 8e-09 & 0.16 & 0.08 \\ \hline 
 624 & -6.3985 &  0.1003 &  0.1840 &  0.5305 &  1.2438 &   48 & 7e-09 & 1e-08 & 0.19 & 0.08 \\ \hline 
   4 & \multicolumn{8}{c|}{\emph{--- Converged to wrong solution ---}} & 0.10 & 0.11 \\ \hline 
   2 & \multicolumn{8}{c|}{\emph{--- Failed to converge ---}} & 0.23 & 0.02 \\ \hline 
 \end{HanResultsSubTable}
\end{table}

For $\beta=1$, Han's method finds all three local maxima, though it fails to converge within 10,000 iterations for four starting points. There are no consistent results with respect to time. Han's method is faster than GEAP for $\lambda=3.7394$ but slower for the other two eigenpairs. This is consistent if we compare the number of function evaluations and the number of iterations for GEAP, which are measuring comparable amounts of work.

For $\beta=-1$, Han's method finds all four local minima, but it fails to converge for two starting points and converges to wrong solutions for four starting points. In those four cases, it terminated because the gradient was small (flag = 1) for three cases and the fourth it stopped because it could no improve the function value (flag = 5). In this case, GEAP was faster on average for all eigenpairs.

In general, Han's method represents an alternative approach to solving the generalized tensor eigenpair problem. In \cite{Ha12}, Han's method was compared to SS-HOPM with a fixed shift (for Z-eigenpairs only) and was superior. However, GEAP is usually as fast as Han's method and perhaps a little more robust in terms of its convergence behavior. The speed being similar is thanks to the adaptive shift in GEAP. It may be that Han's method could avoid problems of converging to incorrect solutions with tighter tolerances, but then the speed would be slower.

\subsection{Other Related Work}

Since \Eqn{gtep} is a polynomial system of equations, we can also
consider a polynomial solver approach. This does not scale to larger
problems and may be slow even for small problems. Nevertheless, it
finds \emph{all} solutions (even saddle points). We have used the Gr\"obner basis polynomial solver
\texttt{NSolve} in Mathematica to compute the full set of solutions
for the problems discussed in this paper.

In terms of methods specifically geared to tensor eigenvalues, most
work has focused on computing the largest H-eigenvalue for a
\emph{nonnegative} tensor: \cite{NgQiZh09,LiZhIb10}. The method of
Liu, Zhou, and Ibrahim \cite{LiZhIb10} is guaranteed to always find
the largest eigenvalue and also uses a ``shift'' approach.

\section{Conclusions}
\label{sec:conclusions}

The paper has proposed two improvements to the SS-HOPM method
\cite{KoMa11}. First, we have adapted the method to the
\emph{generalized} tensor eigenproblem. Second, we have proposed a
method for adaptively and automatically selecting the shift,
overcoming a major problem with the SS-HOPM method because choosing
the shift too large dramatically slows convergence whereas choosing it
too small can cause the method to fail completely.

We have tested our method numerically on several problems from the
literature, including computing of Z-, H-, and D-eigenpairs. We have
also proposed a novel method for generating random symmetric positive
definite tensors.

As this paper was in review, a new method has been proposed to compute \emph{all} real general eigenvalues using Jacobian semidefinite programming relaxations \cite{CuDaNi14}. Comparing to this method will be a topic of future study.

\appendix

\section{Useful derivatives}
\label{sec:derivatives}
First, we consider the gradient and Hessian of the general function
\begin{displaymath}
  f\ParensX = \frac{\FuncA \FuncB}{\FuncC}.
\end{displaymath}
Let $\V{g}_i(\Vx)$ denote $\nabla f_i(\Vx)$. 
From matrix calculus, the gradient of $\Sfx$ is
\renewcommand{\ParensX}{}
\begin{displaymath}
  \V{g}\ParensX = 
  \Bigl(
  \FuncA \GradB 
  + \FuncB \GradA
  - \bigl( \FuncA \FuncB / \FuncC \bigr) \GradC
  \Bigr) 
  \Bigm/ 
  \FuncC.
\end{displaymath}
Here we have dropped the argument, $\Vx$, to simply the notation.
Let $\M{H}_i(\Vx)$ denote $\nabla^2 f_i(\Vx)$. 
The Hessian of $\Sfx$ is
\begin{multline*}
  \M{H} = 
  \frac{\FuncA \FuncB}{\FuncC^3}(\Msop{\GradC}{\GradC})
  + \frac{1}{\FuncC}
  \Bigl[ \FuncB \HessA + \FuncA \HessB + (\Msop{\GradA}{\GradB}) \Bigr]\\
  -\frac{1}{\FuncC^2}
  \Bigl[
  \FuncA \FuncB \HessC
  + \FuncB (\Msop{\GradA}{\GradC}) 
  + \FuncA (\Msop{\GradB}{\GradC}) 
  \Bigr].
\end{multline*}
Now we specialize $\Sfx$ to \Eqn{opt}:
let
$\FuncA = \SAxm$, $\FuncB = \Snxm$, and $\FuncC=\SBxm$.
The following derivatives are proved in \cite{KoMa11}:
\begin{align*}
  \GradA &= m\VAxm, &
  \HessA &= m(m-1)\MAxm, \\
  \GradB & = m \Snxm[m-2] \Vx, &
  \HessB & = m \Snxm[m-2] \M{I} + m(m-2) \Snxm[m-4] \Mxxt, \\
  \GradC &= m \VBxm, &
  \HessC &= m(m-1) \MBxm.
\end{align*}
We need only consider the case for $\Vx \in \Sigma$, so we may assume
\begin{align*}
  \FuncB &= 1, &
  \GradB &= m \Vx, &
  \HessB &= m \bigl( \M{I} + (m-2)\Mxxt \bigr).
\end{align*}
Putting everything together, we have for $\Vx \in \Sigma$,
\begin{displaymath}
  \Vgx 
  = \frac{m}{\SBxm} 
  \Biggl[
  \bigl( \SAxm \bigr) \; \Vx 
  \;+\; \VAxm
  \;-\; \biggl( \frac{\SAxm}{\SBxm} \biggr) \; \VBxm 
  \Biggr].
\end{displaymath}
For the Hessian, assuming $\Vx\in \Sigma$, we have
\begin{multline*}
  \MHx = 
  \frac{m^2 \SAxm}{(\SBxm)^3} \bigl( \Msop{\VBxm}{\VBxm} \bigr) \\
  + \frac{m}{\SBxm} 
  \biggl[
  (m-1) \MAxm 
  + \SAxm \bigl( \M{I} + (m-2) \Mxxt \bigr) 
  + m \bigl( \Msop{\VAxm}{\Vx} \bigr)
  \biggr] \\
  - \frac{m}{(\SBxm)^2} 
  \biggl[
  (m-1) \SAxm \TB\Vx^{m-2}
  + m \bigl( \Msop{\VAxm}{\VBxm} \bigr)
  + m \SAxm \bigl( \Msop{\Vx}{\VBxm} \bigr)
  \biggr].
\end{multline*}

\FloatBarrier
\section{Tensor specifications}
\label{sec:entries}
The tensor for \Ex{KoRe02ex1} comes from Example 1 in Kofidis and Regalia \cite{KoRe02} and is specified in \Fig{KoRe02}.

\begin{figure}[htbp]\footnotesize
  \centering
  \begin{align*}
    a_{1111} &=  \phantom{-}0.2883, & 
    a_{1112} &= -0.0031, & 
    a_{1113} &=  \phantom{-}0.1973, & 
    a_{1122} &= -0.2485,\\
    a_{1123} &= -0.2939, & 
    a_{1133} &=  \phantom{-}0.3847, & 
    a_{1222} &=  \phantom{-}0.2972, & 
    a_{1223} &=  \phantom{-}0.1862,\\
    a_{1233} &=  \phantom{-}0.0919, & 
    a_{1333} &= -0.3619, & 
    a_{2222} &=  \phantom{-}0.1241, & 
    a_{2223} &= -0.3420,\\
    a_{2233} &=  \phantom{-}0.2127, & 
    a_{2333} &=  \phantom{-}0.2727, & 
    a_{3333} &= -0.3054.
  \end{align*}
  \caption{$\T{A}$ from Kofidis and Regalia \cite{KoRe02}, used in \Ex{KoRe02ex1}}
  \label{fig:KoRe02}
\end{figure}

\FloatBarrier
The tensor $\T{A}$ used for \Ex{heig} is randomly generated as described in \Sec{heig}; its entries are specified in \Fig{random_adef}.%

\begin{figure}[htbp]\footnotesize
  \centering
\begin{align*}
a_{111111} &= \phantom{-} 0.2888, & 
a_{111112} &=           -0.0013, & 
a_{111113} &=           -0.1422, & 
a_{111114} &=           -0.0323,\\
a_{111122} &=           -0.1079, & 
a_{111123} &=           -0.0899, & 
a_{111124} &=           -0.2487, & 
a_{111133} &= \phantom{-} 0.0231,\\
a_{111134} &=           -0.0106, & 
a_{111144} &= \phantom{-} 0.0740, & 
a_{111222} &= \phantom{-} 0.1490, & 
a_{111223} &= \phantom{-} 0.0527,\\
a_{111224} &=           -0.0710, & 
a_{111233} &=           -0.1039, & 
a_{111234} &=           -0.0250, & 
a_{111244} &= \phantom{-} 0.0169,\\
a_{111333} &= \phantom{-} 0.2208, & 
a_{111334} &= \phantom{-} 0.0662, & 
a_{111344} &= \phantom{-} 0.0046, & 
a_{111444} &= \phantom{-} 0.0943,\\
a_{112222} &=           -0.1144, & 
a_{112223} &=           -0.1295, & 
a_{112224} &=           -0.0484, & 
a_{112233} &= \phantom{-} 0.0238,\\
a_{112234} &=           -0.0237, & 
a_{112244} &= \phantom{-} 0.0308, & 
a_{112333} &= \phantom{-} 0.0142, & 
a_{112334} &= \phantom{-} 0.0006,\\
a_{112344} &=           -0.0044, & 
a_{112444} &= \phantom{-} 0.0353, & 
a_{113333} &= \phantom{-} 0.0947, & 
a_{113334} &=           -0.0610,\\
a_{113344} &=           -0.0293, & 
a_{113444} &= \phantom{-} 0.0638, & 
a_{114444} &= \phantom{-} 0.2326, & 
a_{122222} &=           -0.2574,\\
a_{122223} &= \phantom{-} 0.1018, & 
a_{122224} &= \phantom{-} 0.0044, & 
a_{122233} &= \phantom{-} 0.0248, & 
a_{122234} &= \phantom{-} 0.0562,\\
a_{122244} &= \phantom{-} 0.0221, & 
a_{122333} &= \phantom{-} 0.0612, & 
a_{122334} &= \phantom{-} 0.0184, & 
a_{122344} &= \phantom{-} 0.0226,\\
a_{122444} &= \phantom{-} 0.0247, & 
a_{123333} &= \phantom{-} 0.0847, & 
a_{123334} &=           -0.0209, & 
a_{123344} &=           -0.0795,\\
a_{123444} &=           -0.0323, & 
a_{124444} &=           -0.0819, & 
a_{133333} &= \phantom{-} 0.5486, & 
a_{133334} &=           -0.0311,\\
a_{133344} &=           -0.0592, & 
a_{133444} &= \phantom{-} 0.0386, & 
a_{134444} &=           -0.0138, & 
a_{144444} &= \phantom{-} 0.0246,\\
a_{222222} &= \phantom{-} 0.9207, & 
a_{222223} &=           -0.0908, & 
a_{222224} &= \phantom{-} 0.0633, & 
a_{222233} &= \phantom{-} 0.1116,\\
a_{222234} &=           -0.0318, & 
a_{222244} &= \phantom{-} 0.1629, & 
a_{222333} &= \phantom{-} 0.1797, & 
a_{222334} &=           -0.0348,\\
a_{222344} &=           -0.0058, & 
a_{222444} &= \phantom{-} 0.1359, & 
a_{223333} &= \phantom{-} 0.0584, & 
a_{223334} &=           -0.0299,\\
a_{223344} &=           -0.0110, & 
a_{223444} &= \phantom{-} 0.1375, & 
a_{224444} &=           -0.1405, & 
a_{233333} &= \phantom{-} 0.3613,\\
a_{233334} &= \phantom{-} 0.0809, & 
a_{233344} &= \phantom{-} 0.0205, & 
a_{233444} &= \phantom{-} 0.0196, & 
a_{234444} &= \phantom{-} 0.0226,\\
a_{244444} &=           -0.2487, & 
a_{333333} &= \phantom{-} 0.6007, & 
a_{333334} &=           -0.0272, & 
a_{333344} &=           -0.1343,\\
a_{333444} &=           -0.0233, & 
a_{334444} &=           -0.0227, & 
a_{344444} &=           -0.3355, & 
a_{444444} &=           -0.5937. \\
\end{align*}
  \caption{$\TA$ from \Ex{heig} and \Ex{random}}
  \label{fig:random_adef}
\end{figure}

The tensor used in \Ex{deig}. The DKI tensor $\TA \in \Smn[4,3]$ (called $\T{W}$ in the original paper) is the
symmetric tensor defined by the unique elements shown in \Fig{deiga}.
The tensor $\T{B}$ is the symmetrized outer product of the matrix
$\M{D}$ with itself where
\begin{displaymath}
  \M{D} = 
  \begin{bmatrix}
    1.755  &  0.035  &  0.132 \\
    0.035  &  1.390  &  0.017 \\
    0.132  &  0.017  &  4.006 
  \end{bmatrix},    
\end{displaymath}
so $\TB$ is the tensor whose unique elements are given in \Fig{deigb}

\begin{figure}[htpb]
  \centering
  \begin{align*}
    a_{1111} &=  \phantom{-}0.4982, & 
    a_{1112} &= -0.0582, & 
    a_{1113} &= -1.1719, & 
    a_{1122} &=  \phantom{-}0.2236,\\
    a_{1123} &= -0.0171, & 
    a_{1133} &=  \phantom{-}0.4597, & 
    a_{1222} &=  \phantom{-}0.4880, & 
    a_{1223} &=  \phantom{-}0.1852,\\
    a_{1233} &= -0.4087, & 
    a_{1333} &=  \phantom{-}0.7639, & 
    a_{2222} &=  \phantom{-}0.0000, & 
    a_{2223} &= -0.6162,\\
    a_{2233} &=  \phantom{-}0.1519, & 
    a_{2333} &=  \phantom{-}0.7631, & 
    a_{3333} &=  \phantom{-}2.6311.
  \end{align*}
  \caption{$\TA$ from Qi, Wang, and Wu \cite{QiWaWu08}, used in \Ex{deig}}
  \label{fig:deiga}
\end{figure}
\begin{figure}[htpb]
  \centering
  \begin{align*}
    b_{1111} &= 3.0800, & 
    b_{1112} &= 0.0614, & 
    b_{1113} &= 0.2317, & 
    b_{1122} &= 0.8140,\\
    b_{1123} &= 0.0130, & 
    b_{1133} &= 2.3551, & 
    b_{1222} &= 0.0486, & 
    b_{1223} &= 0.0616,\\
    b_{1233} &= 0.0482, & 
    b_{1333} &= 0.5288, & 
    b_{2222} &= 1.9321, & 
    b_{2223} &= 0.0236,\\
    b_{2233} &= 1.8563, & 
    b_{2333} &= 0.0681, & 
    b_{3333} &= 16.0480.    
  \end{align*}
  \caption{$\TB$ from Qi, Wang, and Wu \cite{QiWaWu08}, used in \Ex{deig}}
  \label{fig:deigb}
\end{figure}
The tensor $\T{A}$ used in \Ex{random} is the same as is used in \Ex{heig} and specified in \Fig{random_adef}. The tensor $\T{B}$ is a random positive definite tensor from \Ex{random}; it entries are specified in \Fig{random_bdef}.
\begin{figure}[htbp]
  \centering
  \footnotesize
  \begin{align*}
b_{111111} &= \phantom{-} 0.2678, & 
b_{111112} &=           -0.0044, & 
b_{111113} &=           -0.0326, & 
b_{111114} &=           -0.0081,\\
b_{111122} &= \phantom{-} 0.0591, & 
b_{111123} &=           -0.0009, & 
b_{111124} &=           -0.0045, & 
b_{111133} &= \phantom{-} 0.0533,\\
b_{111134} &=           -0.0059, & 
b_{111144} &= \phantom{-} 0.0511, & 
b_{111222} &=           -0.0029, & 
b_{111223} &=           -0.0072,\\
b_{111224} &=           -0.0016, & 
b_{111233} &=           -0.0005, & 
b_{111234} &= \phantom{-} 0.0007, & 
b_{111244} &=           -0.0006,\\
b_{111333} &=           -0.0185, & 
b_{111334} &= \phantom{-} 0.0001, & 
b_{111344} &=           -0.0058, & 
b_{111444} &=           -0.0046,\\
b_{112222} &= \phantom{-} 0.0651, & 
b_{112223} &=           -0.0013, & 
b_{112224} &=           -0.0050, & 
b_{112233} &= \phantom{-} 0.0190,\\
b_{112234} &=           -0.0023, & 
b_{112244} &= \phantom{-} 0.0190, & 
b_{112333} &=           -0.0011, & 
b_{112334} &=           -0.0014,\\
b_{112344} &= \phantom{-}0.0000, & 
b_{112444} &=           -0.0043, & 
b_{113333} &= \phantom{-} 0.0498, & 
b_{113334} &=           -0.0061,\\
b_{113344} &= \phantom{-} 0.0169, & 
b_{113444} &=           -0.0060, & 
b_{114444} &= \phantom{-} 0.0486, & 
b_{122222} &=           -0.0054,\\
b_{122223} &=           -0.0078, & 
b_{122224} &=           -0.0016, & 
b_{122233} &=           -0.0006, & 
b_{122234} &= \phantom{-} 0.0008,\\
b_{122244} &=           -0.0006, & 
b_{122333} &=           -0.0067, & 
b_{122334} &= \phantom{-} 0.0001, & 
b_{122344} &=           -0.0022,\\
b_{122444} &=           -0.0016, & 
b_{123333} &=           -0.0002, & 
b_{123334} &= \phantom{-} 0.0006, & 
b_{123344} &=           -0.0002,\\
b_{123444} &= \phantom{-} 0.0006, & 
b_{124444} &=           -0.0003, & 
b_{133333} &=           -0.0286, & 
b_{133334} &= \phantom{-} 0.0017,\\
b_{133344} &=           -0.0056, & 
b_{133444} &= \phantom{-} 0.0001, & 
b_{134444} &=           -0.0051, & 
b_{144444} &=           -0.0073,\\
b_{222222} &= \phantom{-} 0.3585, & 
b_{222223} &=           -0.0082, & 
b_{222224} &=           -0.0279, & 
b_{222233} &= \phantom{-} 0.0610,\\
b_{222234} &=           -0.0076, & 
b_{222244} &= \phantom{-} 0.0636, & 
b_{222333} &=           -0.0042, & 
b_{222334} &=           -0.0044,\\
b_{222344} &=           -0.0002, & 
b_{222444} &=           -0.0145, & 
b_{223333} &= \phantom{-} 0.0518, & 
b_{223334} &=           -0.0067,\\
b_{223344} &= \phantom{-} 0.0184, & 
b_{223444} &=           -0.0069, & 
b_{224444} &= \phantom{-} 0.0549, & 
b_{233333} &=           -0.0059,\\
b_{233334} &=           -0.0034, & 
b_{233344} &=           -0.0002, & 
b_{233444} &=           -0.0039, & 
b_{234444} &= \phantom{-} 0.0010,\\
b_{244444} &=           -0.0208, & 
b_{333333} &= \phantom{-} 0.2192, & 
b_{333334} &=           -0.0294, & 
b_{333344} &= \phantom{-} 0.0477,\\
b_{333444} &=           -0.0181, & 
b_{334444} &= \phantom{-} 0.0485, & 
b_{344444} &=           -0.0304, & 
b_{444444} &= \phantom{-} 0.2305.    
  \end{align*}
  \caption{$\T{B}$ from \Ex{random}}
  \label{fig:random_bdef}
\end{figure}

\FloatBarrier
\section{Complete lists of real eigenpairs}
\label{sec:eigs}
A polynomial system solver (\texttt{NSolve}) using a Gr\"obner basis
method is available in Mathematica  and has been employed
to generate a complete list of eigenpairs for the examples
in this paper in Tables \ref{tab:KoRe02ex1_ep}--\ref{tab:random_eigs}.
\begin{table}[htbp]
  \caption{All Z-eigenpairs for $\TA\in\Smn[4,3]$ from \Ex{KoRe02ex1}}
  \label{tab:KoRe02ex1_ep}
  \centering
  \footnotesize
  \begin{tabular}{|R|V{2}|E{1}|c|}
    \hline
    \EpHeader{3}{2} \\ \hline
    -1.0954 &  0.5915 & -0.7467 & -0.3043 &  1.86 &  2.75 & Minima \\  \hline
    -0.5629 &  0.1762 & -0.1796 &  0.9678 &  1.63 &  2.38 & Minima \\  \hline
    -0.0451 &  0.7797 &  0.6135 &  0.1250 &  0.82 &  1.25 & Minima \\  \hline
     0.1735 &  0.3357 &  0.9073 &  0.2531 & -1.10 &  0.86 & Saddle \\  \hline
     0.2433 &  0.9895 &  0.0947 & -0.1088 & -1.19 &  1.46 & Saddle \\  \hline
     0.2628 &  0.1318 & -0.4425 & -0.8870 &  0.62 & -2.17 & Saddle \\  \hline
     0.2682 &  0.6099 &  0.4362 &  0.6616 & -1.18 &  0.79 & Saddle \\  \hline
     0.3633 &  0.2676 &  0.6447 &  0.7160 & -1.18 & -0.57 & Maxima \\  \hline
     0.5105 &  0.3598 & -0.7780 &  0.5150 &  0.59 & -2.34 & Saddle \\  \hline
     0.8169 &  0.8412 & -0.2635 &  0.4722 & -2.26 & -0.90 & Maxima \\  \hline
     0.8893 &  0.6672 &  0.2471 & -0.7027 & -1.85 & -0.89 & Maxima \\  \hline
  \end{tabular}
\end{table}   
\begin{table}[h]
  \caption{All H-eigenpairs for $\TA\in\Smn[6,4]$ from \Ex{heig}}
  \label{tab:heig_ep}
  \centering
  \footnotesize
  \begin{tabular}{|R|V{3}|E{2}|c|}
      \hline
      \EpHeader{4}{3} \\ \hline
$-10.7440$ & $ 0.4664$ & $ 0.4153$ & $-0.5880$ & $-0.5140$ & $75.69$ & $30.21$ & $41.28$ & Minima \\  \hline
$-8.3201$ & $ 0.5970$ & $-0.5816$ & $-0.4740$ & $-0.2842$ & $62.11$ & $28.56$ & $15.64$ & Minima \\  \hline
$-4.1781$ & $ 0.4397$ & $ 0.5139$ & $-0.5444$ & $ 0.4962$ & $ 5.67$ & $31.85$ & $21.21$ & Minima \\  \hline
$-3.7180$ & $ 0.6843$ & $ 0.5519$ & $ 0.3136$ & $ 0.3589$ & $26.89$ & $ 7.05$ & $12.50$ & Minima \\  \hline
$-3.3137$ & $ 0.5588$ & $ 0.4954$ & $-0.6348$ & $ 0.1986$ & $-4.83$ & $11.31$ & $17.73$ & Saddle \\  \hline
$-3.0892$ & $ 0.6418$ & $-0.2049$ & $-0.6594$ & $-0.3336$ & $-10.41$ & $22.10$ & $ 6.26$ & Saddle \\  \hline
$-2.9314$ & $ 0.3161$ & $ 0.5173$ & $ 0.4528$ & $-0.6537$ & $31.95$ & $ 6.88$ & $13.47$ & Minima \\  \hline
$-2.0437$ & $ 0.6637$ & $ 0.5911$ & $-0.2205$ & $ 0.4017$ & $15.87$ & $-4.81$ & $ 8.30$ & Saddle \\  \hline
$-1.3431$ & $ 0.0544$ & $ 0.4258$ & $ 0.0285$ & $ 0.9027$ & $ 4.40$ & $ 2.04$ & $-0.85$ & Saddle \\  \hline
$-1.0965$ & $ 0.5156$ & $ 0.3387$ & $ 0.4874$ & $ 0.6180$ & $24.09$ & $14.29$ & $-13.10$ & Saddle \\  \hline
$-1.0071$ & $ 0.2030$ & $ 0.5656$ & $-0.0975$ & $-0.7933$ & $-3.71$ & $ 4.13$ & $ 5.35$ & Saddle \\  \hline
$-0.3600$ & $ 0.6999$ & $-0.1882$ & $ 0.3292$ & $-0.6053$ & $ 9.74$ & $ 3.89$ & $-2.07$ & Saddle \\  \hline
$-0.3428$ & $ 0.3879$ & $-0.1700$ & $ 0.5174$ & $-0.7436$ & $-3.52$ & $ 6.07$ & $ 1.24$ & Saddle \\  \hline
$ 0.0073$ & $ 0.3068$ & $ 0.0539$ & $ 0.3127$ & $-0.8973$ & $-2.92$ & $-1.29$ & $ 1.22$ & Saddle \\  \hline
$ 0.1902$ & $ 0.9744$ & $-0.0316$ & $ 0.2013$ & $-0.0952$ & $-1.49$ & $ 2.17$ & $ 0.65$ & Saddle \\  \hline
$ 0.3947$ & $ 0.5416$ & $ 0.4650$ & $ 0.0708$ & $ 0.6967$ & $ 8.59$ & $-15.89$ & $-3.63$ & Saddle \\  \hline
$ 0.4679$ & $ 0.9613$ & $ 0.0442$ & $-0.2718$ & $ 0.0083$ & $ 1.32$ & $-1.33$ & $-1.73$ & Saddle \\  \hline
$ 0.5126$ & $ 0.4232$ & $-0.6781$ & $-0.2347$ & $ 0.5532$ & $-8.44$ & $ 9.45$ & $ 7.66$ & Saddle \\  \hline
$ 0.5236$ & $ 0.3092$ & $ 0.8725$ & $ 0.1389$ & $-0.3518$ & $-2.58$ & $ 1.68$ & $ 3.60$ & Saddle \\  \hline
$ 0.7573$ & $ 0.5830$ & $-0.2565$ & $-0.3076$ & $-0.7069$ & $ 1.86$ & $-5.35$ & $-14.39$ & Saddle \\  \hline
$ 0.8693$ & $ 0.2414$ & $ 0.8332$ & $-0.2479$ & $-0.4313$ & $ 3.48$ & $-3.31$ & $-2.38$ & Saddle \\  \hline
$ 0.9572$ & $ 0.1035$ & $-0.9754$ & $-0.1932$ & $-0.0221$ & $-2.05$ & $ 0.83$ & $ 1.80$ & Saddle \\  \hline
$ 1.1006$ & $ 0.2033$ & $-0.9035$ & $-0.1584$ & $ 0.3424$ & $ 2.10$ & $-2.38$ & $-1.15$ & Saddle \\  \hline
$ 2.3186$ & $ 0.1227$ & $-0.8044$ & $-0.0334$ & $-0.5804$ & $ 2.50$ & $-2.74$ & $-10.23$ & Saddle \\  \hline
$ 2.7045$ & $ 0.3618$ & $-0.5607$ & $-0.5723$ & $ 0.4766$ & $ 8.78$ & $-17.72$ & $-21.79$ & Saddle \\  \hline
$ 3.3889$ & $ 0.6320$ & $ 0.5549$ & $ 0.3596$ & $-0.4043$ & $16.59$ & $-25.41$ & $-17.68$ & Saddle \\  \hline
$ 3.9099$ & $ 0.6722$ & $-0.2683$ & $-0.1665$ & $ 0.6697$ & $-21.17$ & $-4.98$ & $ 5.01$ & Saddle \\  \hline
$ 4.8422$ & $ 0.5895$ & $-0.2640$ & $-0.4728$ & $ 0.5994$ & $-28.20$ & $-6.48$ & $-15.54$ & Maxima \\  \hline
$ 5.1757$ & $ 0.6513$ & $ 0.0021$ & $ 0.7550$ & $-0.0760$ & $-23.82$ & $ 3.66$ & $-3.35$ & Saddle \\  \hline
$ 5.8493$ & $ 0.6528$ & $ 0.5607$ & $-0.0627$ & $-0.5055$ & $-34.20$ & $-22.87$ & $-9.58$ & Maxima \\  \hline
$ 8.7371$ & $ 0.4837$ & $ 0.5502$ & $ 0.6671$ & $-0.1354$ & $-7.66$ & $-19.48$ & $-43.93$ & Maxima \\  \hline
$ 9.0223$ & $ 0.5927$ & $-0.5567$ & $ 0.5820$ & $-0.0047$ & $-58.03$ & $-28.84$ & $ 4.60$ & Saddle \\  \hline
$ 9.6386$ & $ 0.5342$ & $-0.5601$ & $ 0.5466$ & $-0.3197$ & $-64.78$ & $-41.13$ & $-9.04$ & Maxima \\  \hline
$14.6941$ & $ 0.5426$ & $-0.4853$ & $ 0.4760$ & $ 0.4936$ & $-94.14$ & $-61.11$ & $-54.81$ & Maxima \\  \hline
  \end{tabular}
\end{table}   

\begin{table}[htbp]
  \caption{All D-eigenpairs for $\TA\in\Smn[4,3]$ and $\M{D} \in \Smn[2,3]$ from \Ex{deig}}
  \label{tab:deig}
  \centering
  \footnotesize
  \begin{tabular}{|R|V{2}|E{1}|c|}
    \hline
    \EpHeader{3}{2} \\ \hline
    -0.3313 &  0.2309 & -0.7741 & -0.1509 &  1.02 &  2.11 & Minima \\  \hline
    -0.1242 &  0.6577 &  0.0712 &  0.2189 &  0.35 &  1.25 & Minima \\  \hline
    -0.0074 &  0.2161 &  0.3149 & -0.4485 &  0.36 &  0.46 & Minima \\  \hline
    0.0611 &  0.6113 & -0.4573 &  0.1181 & -0.63 &  1.14 & Saddle \\  \hline
    0.1039 &  0.3314 &  0.5239 &  0.3084 & -0.46 &  0.63 & Saddle \\  \hline
    0.2009 &  0.2440 & -0.1250 &  0.4601 & -0.32 &  0.07 & Saddle \\  \hline
    0.2056 &  0.1211 & -0.2367 & -0.4766 & -0.29 &  0.13 & Saddle \\  \hline
    0.2219 &  0.1143 &  0.1812 &  0.4773 & -0.08 & -0.20 & Maxima \\  \hline
    0.2431 &  0.0943 & -0.6840 &  0.2905 &  0.18 & -1.11 & Saddle \\  \hline
    0.2514 &  0.2485 & -0.5579 &  0.3363 & -0.14 & -0.71 & Maxima \\  \hline
    0.3827 &  0.6236 &  0.3954 & -0.1678 & -1.58 &  0.32 & Saddle \\  \hline
    0.4359 &  0.4336 &  0.6714 & -0.0949 & -0.43 & -1.64 & Maxima \\  \hline
    0.5356 &  0.6638 & -0.1123 & -0.2537 & -0.48 & -1.43 & Maxima \\  \hline
  \end{tabular}
\end{table}      

\begin{table}[htbp]\footnotesize
  \centering
  \caption{All generalized tensor eigenpairs for $\TA\in\Smn[6,4]$ 
    and $\TB\in\Smn[6,4]_+$ from \Ex{random}}
  \label{tab:random_eigs}
  \begin{tabular}{|R|V{3}|E{2}|c|}
      \hline
      \EpHeader{4}{3} \\ \hline
-6.3985 &  0.0733 &  0.1345 &  0.3877 &  0.9090 & 20.43 &  4.93 & 11.20 & Minima \\  \hline
-3.5998 &  0.7899 &  0.4554 &  0.2814 &  0.2991 &  8.05 & 10.39 & 12.41 & Minima \\  \hline
-3.2777 &  0.6888 & -0.6272 & -0.2914 & -0.2174 &  8.27 &  3.65 &  5.95 & Minima \\  \hline
-1.7537 &  0.6329 & -0.2966 & -0.6812 & -0.2180 & -4.25 &  3.00 &  5.56 & Saddle \\  \hline
-1.1507 &  0.1935 &  0.5444 &  0.2991 & -0.7594 &  0.73 &  3.54 &  4.20 & Minima \\  \hline
-1.0696 &  0.1372 &  0.5068 &  0.0665 & -0.8485 & -1.54 &  3.30 &  3.64 & Saddle \\  \hline
-1.0456 &  0.2365 &  0.4798 & -0.7212 &  0.4402 & -1.16 &  1.54 &  2.57 & Saddle \\  \hline
-0.7842 &  0.5409 &  0.3388 &  0.4698 &  0.6099 & 16.02 &  8.79 & -12.47 & Saddle \\  \hline
-0.7457 &  0.6348 &  0.5354 & -0.4388 &  0.3434 &  2.49 &  0.94 & -1.59 & Saddle \\  \hline
-0.2542 &  0.3900 & -0.1333 &  0.4946 & -0.7652 & -2.51 &  2.99 &  0.93 & Saddle \\  \hline
-0.2359 &  0.6956 & -0.1369 &  0.3550 & -0.6094 &  6.38 &  2.23 & -1.27 & Saddle \\  \hline
 0.0132 &  0.3064 &  0.0541 &  0.3111 & -0.8980 & -5.33 & -2.36 &  2.21 & Saddle \\  \hline
 0.1633 &  0.4278 & -0.6578 & -0.2545 &  0.5652 & -2.42 &  3.86 &  2.36 & Saddle \\  \hline
 0.3250 &  0.5265 &  0.4653 &  0.0927 &  0.7055 &  7.50 & -12.05 & -3.41 & Saddle \\  \hline
 0.5206 &  0.3738 & -0.4806 & -0.6066 &  0.5111 &  3.19 & -2.27 & -1.47 & Saddle \\  \hline
 0.5463 &  0.5157 & -0.3055 & -0.3313 & -0.7287 & -9.91 & -3.67 &  1.37 & Saddle \\  \hline
 0.5945 &  0.4015 &  0.8447 &  0.1782 & -0.3058 & -3.70 &  4.95 &  1.87 & Saddle \\  \hline
 0.6730 &  0.9634 & -0.0009 &  0.2396 & -0.1204 & -5.84 &  7.88 &  1.78 & Saddle \\  \hline
 0.8862 &  0.3559 &  0.8571 & -0.1675 & -0.3326 &  3.55 & -2.24 & -2.63 & Saddle \\  \hline
 1.2962 &  0.9849 &  0.0018 & -0.1681 &  0.0419 &  2.20 & -5.97 & -3.18 & Saddle \\  \hline
 1.4646 &  0.7396 &  0.4441 &  0.4009 & -0.3083 &  8.41 & -2.08 & -7.72 & Saddle \\  \hline
 2.9979 &  0.8224 &  0.4083 & -0.0174 & -0.3958 & -4.00 & -5.46 & -6.56 & Maxima \\  \hline
 3.5181 &  0.4494 & -0.7574 &  0.4502 & -0.1469 & -9.40 &  1.89 & -2.83 & Saddle \\  \hline
 3.6087 &  0.0340 & -0.8989 & -0.0373 & -0.4353 &  0.87 & -8.03 & -5.77 & Saddle \\  \hline
 3.7394 &  0.2185 & -0.9142 &  0.2197 & -0.2613 & -8.72 & -0.90 & -3.34 & Maxima \\  \hline
11.3476 &  0.4064 &  0.2313 &  0.8810 &  0.0716 & -7.20 & -18.98 & -21.53 & Maxima \\  \hline    
  \end{tabular}
\end{table}

\FloatBarrier
\bibliographystyle{siamdoi}

 
\end{document}